\documentclass[]{interact}

\usepackage{color}

\usepackage{epstopdf}% To incorporate .eps illustrations using PDFLaTeX, etc.
\usepackage[caption=false]{subfig}% Support for small, `sub' figures and tables

\usepackage[numbers,sort&compress]{natbib}% Citation support using natbib.sty
\bibpunct[, ]{[}{]}{,}{n}{,}{,}% Citation support using natbib.sty
\renewcommand\bibfont{\fontsize{10}{12}\selectfont}% Bibliography support using natbib.sty
\makeatletter% @ becomes a letter
\def\NAT@def@citea{\def\@citea{\NAT@separator}}% Suppress spaces between citations using natbib.sty
\makeatother% @ becomes a symbol again

\newcommand{\be}{\begin{equation}}
\newcommand{\ee}{\end{equation}}
\newcommand{\bea}{\begin{eqnarray}}
\newcommand{\eea}{\end{eqnarray}}

\newcommand{\ep}{\varepsilon}
\newcommand{\bst}{\boldsymbol{\xi}}
\newcommand{\cop}{{\cal COP}(p)}

\newcommand{\cp}{{\cal CP}(p)}

\theoremstyle{plain}% Theorem-like structures provided by amsthm.sty
\newtheorem{theorem}{Theorem}[section]
\newtheorem{lemma}[theorem]{Lemma}

\newtheorem{proposition}[theorem]{Proposition}

\theoremstyle{definition}
\newtheorem{definition}[theorem]{Definition}

\theoremstyle{remark}
\newtheorem{remark}{Remark}

\begin{document}

%\articletype{Original article}% Specify the article type or omit as appropriate

\title{Properties of the complementarity set for the  cone of copositive matrices}

\author{
\name{O.~I. Kostyukova\textsuperscript{a}\thanks{CONTACT O.~I. Kostyukova. Email: kostyukova@im.bas-net.by}}
\affil{\textsuperscript{a}Institute of mathematics, Nationaly academy of sciences of Belarus, Minsk, Belarus}
 %\textsuperscript{b}Institut f\"{u}r Informatik, Albert-Ludwigs-Universit\"{a}t, Freiburg, Germany}
}

\maketitle

\begin{abstract}
For a proper cone $K$ and its dual cone $K^*$ in $\mathbb R^n$, the complementarity set of $K$ is defined as
 ${\mathbb C}(K)=\{(x,y): x\in K,\; y\in K^*,\, x^\top y=0\}$. It is known that ${\mathbb C}(K)$ is an $n$-dimensional manifold in the space
$\mathbb R^{2n}$. If $ K$ is a symmetric cone, points in ${\mathbb C}(K)$  must satisfy at least $n$ linearly
 independent bi-linear identities. Since this knowledge comes in handy when optimizing over such cones,
it makes sense to search for similar relationships for non-symmetric cones.

In this paper, we  study  properties of the complementarity set for the dual cones of copositive and completely positive matrices.
Despite these cones  are of great interest  due to their applications in
 optimization, they have not yet been sufficiently studied.

\end{abstract}


\begin{amscode} 49N15, 90C25,  90C33, 90C46
\end{amscode}

\section{Introduction}\label{S-1}

Let ${\cal K}$ be a proper cone in $ \mathbb R^n$ (that is, a closed, pointed and convex cone with nonempty interior in $ \mathbb R^n$) and
${\cal K}^*$ be its dual cone. The set
\be {\mathbb  C}({\cal K}):=\{(x, s) : x \in {\cal K},\  s\in {\cal K}^*, \ x^\top s = 0\}\label{C}\ee
is called the complementarity set of ${\cal K}$.
It follows from   the definition that
${\mathbb  C}({\cal K})$ and ${\mathbb  C}({\cal K}^*)$ are congruent: one can be obtained from the other by exchanging
the first and last $n$ coordinates.

The complementarity sets play  a very significant role
 in the context of primal and dual linear optimization
problems over a cone \cite{anjos2011handbook, letchford2018guide, pataki2000geometry, ramirez2020refining, mohammad2020parametric,
 bonnans2005perturbation, kostyukova2020immobile}
 and in complementarity problems \cite{cottle2009linear, gao2022monotone}.
The strict complementarity condition plays a central role in  establishing  error bounds and quantifying the sensitivity of the solution of conic problems
\cite{ding2023strict}.

The following theorem was proved in \cite{alizadeh1997optimization, rudolf2011bilinear}
\begin{theorem} For any proper cone ${\cal K}$  in $ \mathbb R^n$, the complementarity set ${\mathbb  C}({\cal K})$ is an
$n$-dimensional manifold.
\end{theorem}

In other words, there are $n$ linearly independent  functions $f_i(x, s) $ for $i = 1,...,n$ such that  in the presence of cone constraints
$ x \in {\cal K}$, $ s\in {\cal K}^*$, one complementarity  condition $x^\top s=0$ is equivalent to   $n$ conditions
$f_i(x, s) = 0 $ for $i = 1,...,n.$ These $n$ equalities are the complementary relations (or conditions).
 The $n$ functions are not unique and are characterized by the cone ${\cal K}$.

If these functions are known, then solving a conic optimization problem over the cone ${\cal K}$, we can  combine 
the complementary equations with primal and dual feasibility equations, and get a system of equations with
equal number of unknowns and equations. In the absence of various forms of degeneracy, this set of equations 
determines the primal and dual optimal solutions. Hence, the problem of explicit defining such functions is of great importance.

For some ${\cal K}\subset \mathbb R^n,$ the manifold ${\mathbb{C}}({\cal K})$  can be described
by a set of $n$  rather simple functions $f_i(x, s), $  $i = 1,...,n,$ w.r.t. $x$ and $s$, namely by $n$ bi-linear functions
$x^\top Q_is,$ $i=1,...,n,$ where $ \{Q_i,$ $i=1,...,n\}$ is a set of linearly independent $n\times n$  matrices.
 Such cones are called {\it perfect} (see \cite{gowda2014bilinearity}). This fact proves to be very useful when
optimizing over such cones.
It is known \cite{rudolf2011bilinear, gowda2014bilinearity} that a symmetric cone   ${\cal K}$   (that is, ${\cal K}$  is self-dual and its automorphism group acts
transitively on its interior) is perfect. Some generalizations of the second-order cone, which are perfect, are considered in
\cite{gao2022monotone, sznajder2016lyapunov}.

In general, for non-symmetric cones, the complementarity conditions cannot be represented only by bi-linear functions
 \cite{rudolf2011bilinear, gowda2014bilinearity}.
Hence, it is natural to  try to find other types of functions, characterizing the complementarity conditions for  non-symmetric cones.
Despite the fact that such functions  exist, constructing them explicitly is not an easy task.
This will explain the fact that to date there are few known works devoted to this problem \cite{rudolf2011bilinear}.

In this paper, we examine the complementarity conditions
for the cones $\cop$ and $\cp$  of  copositive  and
  completely positive ${p\times p}$ matrices.
The cones $\cop$ and $\cp$ are  proper cones (i.e. closed, convex, pointed, and full dimensional) and  they are dual to each other.
These cones
 are of great importance due to their applications in
optimization, especially in creating convex formulations of NP-hard problems \cite{anjos2011handbook, bomze2012copositive, bomze2012think, berman2015open}.

Despite their popularity in applications and the large number of studies devoted to these cones, they are still not well studied \cite{berman2015open}.
The cones $\cop$ and $\cp$ are   complicated,  they do not possess some ''good'' properties of the conic sets:
they are neither self-dual, nor homogeneous, nor nice, and hence not facially exposed.

Notice that both cones $\cop$ and $\cp$ have  dimension $p(p+1)/2$.
It is known \cite{gowda2014bilinearity} that there exist only $p$ linearly independent bi-linear functions $f_k(X,U),$  $X\in \cop,$ $U\in \cp,$ $k=1,...,p$,
which are sequences of the complementary condition ${\rm trace}(XU)=0.$
Hence, for these cones,  the complementarity system can never be written as a square system by means of bi-linear
functions alone, and it is necessary to look for other types of functions.

In the paper for the cone ${\cal K}\subset \mathbb R^{n\times n}$ under consideration, using a pair $(x^0,s^0)\in {\mathbb C}({\cal K}), $
we introduce a set of $m$ bi-linear functions  $\Omega_i(x,w)$, $i=1,...,m$, and $n$ linear functions $s=Lw$ w.r.t.
 extended set of variables $x\in \mathbb R^n,$ $s\in \mathbb R^n,$  $w\in \mathbb R^m,$ that satisfy the condition
$${\rm rank}\begin{pmatrix}
\frac{\partial \Omega_i(x^0,w^0)}{\partial (x,w)}
\end{pmatrix}=m \mbox{ where }w^0=Ls^0,$$
and allow us to completely describe the set ${\mathbb C}({\cal K})$ in a neighborhood of a point $(x^0,s^0)$ under some non-degeneracy assumptions.

The paper is structured as follows. Sections \ref{S-1} and \ref{S-2} contain an introduction,  and basic notations and problem statement.
    In Section \ref{S-3}, some auxiliary results are formulated and proved.
 In Section \ref{S-4}, we study  properties of the complementarity set for the cone $\cop$, and based on these properties we
 obtain a system of equations called defining equations. This system allows us to describe the complementarity set
in a neighborhood  of a given point $(X^0,U^0)\in \mathbb C(\cop)$ under some  assumptions.
In Section \ref{S-4}, we show that under made assumptions, the defining equations are independent
in a neighborhood  of the point $(X^0,U^0).$ Section \ref{S-5}  contains some examples that
 show that  assumptions made in the paper  are essential for our study.  Technical propositions are proved in appendix.

\section{Problem statement}\label{S-2}
Let $p$ be a positive integer and $P=\{1,...,p\}$. We denote by $\mathbb R^p$ the $p$-dimensional Euclidean vector space with standard orthogonal basis $\{e_k, k=1,...,p\}$.
We use $ \mathbb R^p_+$ to denote
the set of element-wise non-negative $p$-vectors.
 For $t=(t_k,k\in P)^\top\in \mathbb R^p,$ we denote its support as by ${\rm supp}(t):=\{k\in P:t_k\not=0\}$ and use the norm
$||t||_1=\sum\limits_{k=1}^p|t_k|.$

We let $\mathbb S(p)$  denote the real linear space of
 symmetric  $ p\times p$ matrices. For $U\in \mathbb S(p)$ and $W\in \mathbb S(p)$,  an inner product  is defined by $U\bullet W=
{\rm trace }(UW)$.    For a set ${\cal L}\subset \mathbb S(p)$, we denote by ${\rm int}\, {\cal L},$ ${\rm ri}\, {\cal L}$, ${\rm conv }\, {\cal L}$,
 ${\rm cone }\, {\cal L}$
the interior, relative interior, convex hull, conic hull  of ${\cal L}$, respectively,  and by  ${\rm span}\, {\cal L}$  the space spanned by ${\cal L}$ and
 by ${\cal L}^\bot$ the orthogonal complement of its
span.
 We use  $\mathbb O_{n\times p}$, $E(p)$,  and  ${\bf{0}}$   to denote  the $n\times p$ matrix  of all zeroes, identity  $p\times p$ matrix, and a null  vector.
 The  dimension    of the vector  will be clear from the context.

For a  matrix $F\in \mathbb  S(p)$, we denote by  $F_{ij},$ $i\in P,$ $j\in P,$ its elements and  define the vector
${\rm svec}(F)\in \mathbb R^{p(p+1)/2}$ by the rule (see \cite{alizadeh1998primal})
$$ {\rm svec}(F)=(F_{11},\sqrt{2}F_{21}, ... ,\sqrt{2}F_{p1}, F_{22},\sqrt{2}F_{32}, ...,\sqrt{2}F_{p2}, ... , F_{pp})^\top.$$%\label{svec}\ee
For a set of symmetric $p\times p $ matrices $F(i),$ $i \in I,$ we consider that
$ {\rm rank }(F(i),\, i \in I):=$ ${\rm rank }({\rm svec}(F(i)),\, i \in I).$

For a cone ${\cal K}\subset \mathbb  S(p)$, we denote by ${\cal K}^*$ its dual cone
$${\cal K}^*:=\{U\in \mathbb  S(p): U\bullet X\geq 0\ \forall X\in {\cal K}\}.$$

 We use $ \mathbb  S_+(p)$ to denote the cone of   symmetric semidefinite $ p\times p$ matrices and
$\cop$ to denote the cone of copositive $ p\times p$ matrices
$$\cop:=\{D\in \mathbb  S(p):t^\top Dt\geq 0\ \forall t \in \mathbb R^p_+\}.$$
It is easy to see that the cone $\cop$ can be equivalently  defined as follows:
$$\cop:=\{D\in \mathbb  S(p):t^\top Dt\geq 0\ \forall t \in T\} \mbox{ where } T:=\{t \in \mathbb R^p_+:||t||_1=1\}.$$

 The cone $\cop$  is not self-dual,
its dual cone is the cone of completely positive $ p\times p$  matrices
$$\cp:= (\cop)^*={\rm cone }\{t\,t^\top : t \in \mathbb R^p_+\}.$$

We say that   matrices $X\in \cop$ and $U\in \cp$ are  complementary if $X\bullet U=0.$

\begin{definition} (see \cite{pataki2000geometry}) Let  $X^0\in \cop$ and $U^0\in \cp$, we say that $X^0$ is strictly complementary to $U^0$ if
\be X^0\in {\rm ri}(\cop\cap {U^0}^\perp),\label{X-strict}\ee
 $U^0$ is strictly complementary to $X^0$ if
\be U^0\in {\rm ri}(\cp\cap {X^0}^\perp). \label{U-strict}\ee
\end{definition}

For the cone $\cop$ consider the complementarity set
\bea &\mathbb C({\cop}):=\{(X,U)\in {\cop}\times ({\cop})^*:X\bullet U=0\}=\nonumber\\
&\{(X,U):X\in {\cop},\,U\in  {\cp},\,X\bullet U=0\}.\nonumber\eea

The purpose of this paper is to obtain explicit relations in the form of equations
 that allow us to characterize the complementarity  set $\mathbb C({\cop})$ in the neighborhood of a given point
$(X^0,U^0)\in \mathbb C({\cop})$ satisfying some conditions.

\section{Auxiliary constructions and results}\label{S-3}
For a given matrix $X^0\in\cop$, denote by $T_a(X^0)$ the set of its normalized zeroes
$$T_a(X^0):=\{t \in T:t^\top X^0t=0\}.$$
The set $T_a(X^0)$ is empty or a union of a finite number of convex bounded polyhedra \cite{kostyukova2022equivalent}.

Suppose that $T_a(X^0)\not=\emptyset.$ Denote  by $\tau(j), j \in J\subset \mathbb N,$  the set of all vertices of the set ${\rm conv}\, T_a(X)$
and define the corresponding index sets
\be M(j):=\{k \in P:e^\top_kX^0\tau(j)=0\}, j \in J.\label{def-M(j)}\ee

Having $\tau(j), M(j), j \in J,$ let us partition the index set $ J $ into  maximum number of  subsets $J(s)$, $ s\in S\subset \mathbb N,$ $|S|\geq 1$, such that\\
{\bf a)} $J=\bigcup\limits_{s\in S}J(s),$  {\bf b)}  for any $s \in S,$ it holds  $\bigcup\limits_{j\in J(s)} {\rm supp}(\tau(j))\subset M(i) \;
 \forall i \in J(s),$
{\bf c)} if $|S|\geq 2,$ then for all $  s \in S,$ $\bar s \in S,$ $s\not =\bar s,$ we have
$J(s) \setminus  J(\bar s)\not=\emptyset, \ J(\bar s) \setminus J(s) \not=\!\emptyset,$
and $\ \forall    i_0\in J(s)\setminus J(\bar s),\ \exists\, j_0\in J(\bar s)\setminus J(s)\,\mbox{ and }
\exists k_0\in {\rm supp}(\tau(i_0))$  such that $ k_0 \not \in  M(j_0).$

The following proposition is proved in \cite{kostyukova2021structural}.

\begin{proposition}\label{P1}  Let $\{J(s), s \in S\}$ be the partition of the set $J$ such that the conditions {\bf a)} - {\bf c)} are satisfied.
Then the set $T_{a}(X^0)$ can be presented in the
form
\be T_{a}(X^0)=\bigcup\limits_{s \in S}T_{a}(s,X^0), \mbox{ where } T_a(s,X^0):={\rm conv} \{\tau(j), j \in J(s)\}, \ \forall s \in S.\label{part}\ee
\end{proposition}

Denote
\be P_*(s)=P_*(s,X^0):=\bigcup\limits_{j\in J(s)}{\rm supp}(\tau(j))\subset P,\;  s \in S.\label{P*}\ee
It follows from condition {\bf b)} that
$$P_*(s)\subset M(j)\ \forall j \in J(s),\ \forall s\in S.$$

Consider a matrix $U^0\in \cp$ such that  $(X^0,U^0)\in \mathbb C({\cop})$.
If  $U^0\in {\cp}$ and $X^0\bullet U^0=0$, then the matrix  $U^0$ admits a presentation
$$U^0=\sum\limits_{i \in I^*}\alpha_it(i)(t(i))^\top\mbox{ with some } \alpha_i>0,\, t(i)\in T_a(X^0),\; i \in I^*.$$
Consequently, $U^0$ admits a presentation
\be U^0=\sum\limits_{s \in S}U^0(s),\ U^0(s)\in {\cal F}(s):={\rm cone}\{t\, t^\top, t \in T_a(s,X^0)\}\subset \cp\ \forall s\in S.\label{1}\ee
Notice that by construction
\be U^0_{kq}(s)=0\ \forall k \in P\setminus P_*(s),\ \forall q\in P,\ \forall s\in S.\label{1.10}\ee
Also notice that, in general, there may  exist another set of matrices $\bar U^0(s)\in {\cal F}(s), s \in S,$ such that $U^0$ admits a presentation
$ U^0=\sum\limits_{s \in S}\bar U^0(s).$

Denote
\bea&V(I)=\{(i,j): i \in I,\, j \in I,\, i\leq j\} \mbox{ for } I\subset \mathbb N,\nonumber\\
&\bst(i,j)=(\tau(i)+\tau(j)) \mbox{ for } (i,j)\in V(J(s)), \ s \in S.\nonumber\eea
For $s \in S,$ let  $J_b(s)\subset J(s)$ be such that
$${\rm rank}(\tau(j), j \in J(s))={\rm rank}(\tau(j), j \in J_b(s))=|J_b(s)|.$$
It is known (see \cite{dickinson2011geometry}) that for any $s\in S,$ the
matrices $ \bst(i,j)\bst(i,j)^\top,\  (i,j)\in V(J_b(s))$
 are linearly independent.

One can show that
\be {\rm rank}(t\,t^\top, t \in T_a(s,X^0))={\rm rank}\big(\bst(i,j)\bst(i,j)^\top,\  (i,j)\in V(J_b(s))\big)=|V(J_b(s))|,\label{rank1}\ee
and the following lemma holds true.
\begin{lemma}\label{L-1} Condition (\ref{U-strict}) is equivalent to the following one:
 matrix $U^0$ admits a presentation (\ref{1}) where matrices $U^0(s),s \in S,$ have the form
 \bea&  U^0(s)=\sum\limits_{(i,j)\in V(J(s))}\alpha_{ij}\bst(i,j)(\bst(i,j))^\top+\sum\limits_{i \in I(s)}t(i)(t(i))^\top,\label{A10}\\
&\alpha_{ij}>0, (i,j)\in V(J(s));\ t(i)\in {\rm cone} \,T_a(s,X^0), \; i \in I(s), \; s\in S. \nonumber\eea
\end{lemma}

In what follows, we will say that a pair $(X^0,U^0)\in \mathbb C(\cop)$ satisfies

\vspace{2mm}

Assumption j), if $U^0$ is strictly complementary to $X^0$,

Assumption jj), if the following matrices are linearly independent
\be (\tau(i)+\tau(j))(\tau(i)+\tau(j))^\top, \ (i,j)\in V(J_b(s)),\, s \in S,\label{MM2}\ee

Assumption jjj), if $M(j)=P_*(s)\ \forall j \in J(s), \ \forall s \in S.$

\begin{remark} One can show that Assumptions j) and jjj) imply that  $X^0$ is strictly complementary to $U^0$.
\end{remark}

For $s\in S,$ using given set $P_*(s)$  and number $ p(s):=|P_*(s)|,$
let us introduce  matrix transformations
\be {\cal A}(X,s):{\mathbb S}(p)\to {\mathbb S}(p(s)) \ \mbox{ and } \ {\cal B}(W,s):{\mathbb S}(p(s))\to {\mathbb S}(p)\label{trans}\ee defined  by the rules
\bea& {\cal A}(X,s)=:X(s)=
\begin{pmatrix}
 x_{kq} , q\in P_*(s)\cr
k\in P_*(s)\end{pmatrix}, {\cal B}(W,s)=:U(s)=\begin{pmatrix}
 u_{kq} , q\in P\cr
k\in P\end{pmatrix},\nonumber\\
 & u_{kq}= w_{kq} \mbox{ for }  k \in  P_*(s)\mbox{ and }  q \in  P_*(s), \
 u_{kq}=0\mbox{ for } k \in P\setminus  P_*(s)\mbox{ and } q \in  P,\nonumber\eea
where $x_{kq}$ is $(k,q)$-th element of matrix $X$ and $w_{kq}$ is $(k,q)$-th element of matrix $W.$

\begin{proposition}\label{P-A}
Suppose that for a pair $(X^0,U^0)\in \mathbb C({\cal COP})$  Assumptions j) and jj) hold true.
Then  the following conditions are satisfied: \\
 i) there exists the only set of matrices $U^0(s), $ $ s \in S$, satisfying (\ref{1}),\\
ii) for any $s \in S$, matrix
$W^0(s):={\cal A}(U^0(s),s)$ admits a presentation
\be W^0(s)={\cal M}(s)({\cal M}(s))^\top \mbox{ with some } {\cal M}(s)\in \mathbb R^{p(s)\times n(s)},\
{\cal M}(s)> \mathbb O_{p(s)\times n(s)}, \label{W-M}\ee
iii) ${\cal A}(X^0,s)\!\in\!{\mathbb S}_+(p(s)), \, {\cal A}(U^0(s),s)\!\in \!{\mathbb S}_+(p(s)),$
 ${\cal A}(X^0,s)+{\cal A}(U^0(s),s)\!\in \!{\rm int}\,{\mathbb S}_+(p(s))\,
 \forall s \in S.$
\end{proposition}
\begin{proof} In fact, it is easy to show that Assumption jj) implies condition i).

To prove conditions ii) and iii), let us consider a fixed $s\in S$ and
 denote $X^0(s):={\cal A}(X^0,s),$ $W^0(s):={\cal A}(U^0(s),s).$
Due to  Lemma \ref{L-1},  the matrix $W^0(s)$ admits a presentation
\bea & W^0(s)=\sum\limits_{(i,j)\in V(J(s))}\bst_*(i,j)\bst_*(i,j)^\top +\sum\limits_{i \in I(s)}t_*(i)(t_*(i))^\top \label{W(s)}\\
&\mbox{where } \qquad  \qquad   \qquad  \tau_*(j):=(e^\top_k\tau(j), k \in P_*(s))\in \mathbb R^{p(s)}_+,\, j \in J(s),\qquad  \qquad  \qquad \nonumber\\
&\alpha_{ij}>0 , \,\bst_*(i,j):=\sqrt{\alpha_{ij}}(\tau_*(i)+\tau_*(j))
\in \mathbb R^{p(s)}_+,\; (i,j)\in V(J(s)),\,\nonumber \\
& t_*(i)\in T_*(s,X^0):={\rm cone}\{\tau_*(j), j \in J(s)\}\subset \mathbb R^{p(s)}_+,\, i \in I(s).\nonumber\eea

Let us prove conditions ii). Denote
$$\hat t=\sum\limits_{(i,j)\in V(J(s))}\bst_*(i,j) +\sum\limits_{i \in I(s)}t_*(i).$$
By construction, $\hat t>{\bf 0}$ and $\hat t\in T_*(s,X^0)$. It is easy to check that
\bea&\sum\limits_{(i,j)\in V(J(s))}(\bst_*(i,j)+\theta \hat t)(\bst_*(i,j)+\theta \hat t)^\top +
\sum\limits_{i \in I(s)}(t_*(i)+\theta \hat t)(t_*(i)+\theta \hat t)^\top \label{W0}\\
&=W^0(s)+2\theta \hat t \hat t^\top+\theta ^2\gamma  \hat t \hat t^\top,  \ \gamma:=| V(J(s))|+|I(s)|.\nonumber\eea

Note that equalities   (\ref{rank1}) imply the equalities
\bea&\!\! {\rm rank}(t t^\top, t \in T_*(s,X^0))\!=\!{\rm rank}\big(\bst_*(i,j)\bst_*(i,j)^\top\!\!,\,
 (i,j)\!\in\! V(J_b(s))\big)\!=\!| V(J_b(s))|.\label{rank2}\eea
Hence,  for sufficiently small $\theta >0$, we have
\bea& {\rm rank}(t t^\top, t \in T_*(s,X^0))={\rm rank}\big((\bst_*(i,j)+\theta \hat t)(\bst_*(i,j)+\theta \hat t)^\top\!,\,
 (i,j)\in V(J_b(s))\big).\nonumber\eea
Taking into account this equality, equalities (\ref{rank2}) and the inclusion $\hat t\in T_*(s,X^0)$, we obtain that
\bea& \hat t\hat t^\top=\sum\limits_{(i,j)\in V(J_b(s))}\beta_{ij}(\theta)(\bst_*(i,j)+\theta \hat t)(\bst_*(i,j)+\theta \hat t)^\top,\label{31}\\
&\hat t\hat t^\top=\sum\limits_{(i,j)\in V(J_b(s))}\beta_{ij}\bst_*(i,j)\bst_*(i,j)^\top,\
\beta_{ij}(\theta)=\beta_{ij}+O(\theta), \, (i,j)\in  V(J_b(s)).\nonumber\eea
Hence, it follows from (\ref{W0}) and (\ref{31}) that
\bea &W^0(s)=\sum\limits_{(i,j)\in V(J(s))}(\bst_*(i,j)+\theta \hat t)(\bst_*(i,j)+\theta \hat t)^\top +
\sum\limits_{i \in I(s)}(t_*(i)+\theta \hat t)(t_*(i)+\theta \hat t)^\top - \nonumber\\
&(2\theta +\theta ^2\gamma )\sum\limits_{(i,j)\in V(J_b(s))}\beta_{ij}(\theta)(\bst_*(i,j)+\theta \hat t)(\bst_*(i,j)+\theta \hat t)^\top = \nonumber\\
&\sum\limits_{(i,j)\in V(J(s))}\mu_{ij}(\theta)(\bst_*(i,j)+\theta \hat t)(\bst_*(i,j)+\theta \hat t)^\top +
\sum\limits_{i \in I(s)}(t_*(i)+\theta \hat t)(t_*(i)+\theta \hat t)^\top\nonumber\eea
where $\mu_{ij}(\theta)=1-(2\theta +\theta ^2\gamma )\beta_{ij}(\theta),$ $ (i,j)\in V(J_b(s)),$ $\mu_{ij}(\theta)=1,$ $ (i,j)\in V(J(s))\setminus V(J_b(s)).$
Notice that for sufficiently small $\theta>0$ the following inequalities hold true
$$\mu_{ij}(\theta)>0, \, \bst_*(i,j)+\theta \hat t>{\bf 0} \ \forall  (i,j)\in V(J(s)); \ t_*(i)+\theta \hat t>{\bf 0}\ \forall i \in I(s).$$
Hence we have shown that the matrix $W^0(s)$ admits a presentation
(\ref{W-M}) with
$${\cal M}(s)=\big(\sqrt{\mu_{ij}(\theta)}(\bst_*(i,j)+\theta \hat t), (i,j)\in V(J(s)); \;  t_*(i)+\theta \hat t, i \in I(s)\big).$$

Let us prove conditions iii).
By construction, $W^0(s)\in {\cal CP}(p(s))$, $X^0(s)\in {\cal COP}(p(s))$ and $X^0(s)\bar t={\bf 0}$ where
 $\bar t=\sum\limits_{j \in J(s)}\tau_*(j)\in \mathbb R^{p(s)}$, $\bar t>{\bf 0}.$
Hence
$$X^0(s)\in {\mathbb S}_+(p(s)),\ W^0(s)\in {\mathbb  S}_+(p(s))\ \Longrightarrow\ X^0(s)+W^0(s)\in {\mathbb  S}_+(p(s)).$$
Let  $\tilde t \in \mathbb R^{p(s)}$ be such that
$\tilde t^\top(X^0(s)+W^0(s))\tilde t=0 $. This implies
$$ \tilde t^\top X^0(s)\tilde t =0, \; \tilde t^\top W^0(s)\tilde t =0 \; \Longrightarrow\;
 X^0(s)\tilde t ={\bf 0}, \;  W^0(s)\tilde t ={\bf 0}.$$
If $X^0(s)\tilde t ={\bf 0}$, then $\tilde t\in H(X^0(s)):=\{t \in \mathbb R^{p(s)}:X^0(s)t={\bf 0}\}.$ Note  that
 $X^0(s)\in {\mathbb  S}_+(p(s))$ and $X^0(s)\bar t={\bf 0},$ $\bar t>{\bf 0}.$ Hence it follows from
 Proposition \ref{PPA2} that
\be H(X^0(s))={\rm span}\{\tau_*(j), j \in J(s)\}.\label{Hs}\ee
This implies that $\tilde t=\sum\limits_{j \in J(s)}\beta_j\tau_*(j).$ Then taking into account presentation (\ref{W(s)}), we obtain
$0=\tilde t^\top W^0(s)\tilde t =\sum\limits_{(i,j)\in V(J(s))}\alpha_{ij}(\tilde t^\top(\tau_*(i)+\tau_*(j)))^2+
 \sum\limits_{i \in I(s)}(\tilde t^\top t_*(i))^2.$
Consequently, $$\tilde t^\top(\tau_*(i)+\tau_*(j))=0\ \forall (i,j)\in V(J(s)) \ \Longrightarrow\ \tilde t^\top\tau_*(j)=0\ \forall j \in J(s).$$
It follows from the latter equalities and (\ref{Hs}) that $\tilde t\in (H(X^0(s)))^\bot$.
This inclusion  and inclusion $\tilde t\in (H(X^0(s)))$ proved above imply the  equality $\tilde t={\bf 0}$.
Thus we have proved that condition iii) holds true. \end{proof}

To prove the next proposition, we will use the following lemma
proved in \cite{groetzner2020factorization}.
\begin{lemma} \label{dur} Let ${\cal M},\, {\cal L}\in \mathbb R^{p\times m}. $ Then  ${\cal L}{\cal L}^\top={\cal M}{\cal M}^\top$ iff
there exists an orthogonal matrix $\Omega\in \mathbb R^{m\times m} $ such that ${\cal L}\Omega={\cal M},$  ${\cal L}={\cal M}\Omega^\top.$
\end{lemma}

\begin{proposition}\label{P-F} Let $\ep_0>0$ be a sufficiently small number.
Suppose that for  $s\in S,$ matrix  $W^0(s)\in {\cal CP}(p(s))$  satisfies relations (\ref{W-M}) and
 matrices $W(s,\ep)\in {\mathbb S}_+(p(s)), $ $\ep\in [0,\ep_0],$  are such that
$W(s,\ep)\to W^0(s) \mbox{ as } \ep \to 0.$
Then for $s \in S$ and $\ep\in [0,\ep_0],$  the matrix  $W(s,\ep)$  admits a presentation
\be \!\!W(s,\ep)= {\cal M}(s,\ep)( {\cal M}(s,\ep))^\top \mbox{ with }
  {\cal M}(s,\ep)\in \mathbb R^{p(s)\times n(s)}, \;  {\cal M}(s,\ep)> \mathbb O_{p(s)\times  n(s)}.\label{11}\ee
\end{proposition}
\begin{proof} Let us consider a fixed $s \in S.$
By assumption
$W(s,\ep)\in {\mathbb  S}_+(p(s)),$ $\ep\in [0,\ep_0].$
Hence  there exist matrices $B(s,\ep)\in \mathbb R^{p(s)\times p(s)}$  such that
$||B(s,\ep)||<\infty,$ $W(s,\ep)=B(s,\ep)B^\top(s,\ep)$ for $ \ep\in [0,\ep_0]$.
Let $B(s,0)$  be any limit point of the sequence $B(s,\ep)$ as $\ep\to 0.$  Then taking into account that $W(s,\ep)\to W^0(s)$ as $\ep\to 0,$ we obtain that
$W^0(s)=B(s,0)B^\top(s,0).$

On the other hand, it follows from relations (\ref{W-M})  that
$W^0(s)={\cal M}(s){\cal M}^\top(s),$
where ${\cal M}(s)\in \mathbb R^{p(s)\times n(s)} , \ {\cal M}(s)>\mathbb O_{p(s)\times n(s)}.$
Without loss of generality, we may suppose that $n(s)\geq p(s)$ since otherwise one can extend the  matrix
${\cal M}(s)$ with columns ${\cal M}_i(s),$ $ i=1,...,n(s),$ by  replacing  one  column ${\cal M}_1(s)$ with $p(s)-n(s)+1$ identical
 columns ${\cal M}_1(s)/\sqrt{p(s)-n(s)+1}.$

 Thus we have
\bea &W(s,\ep)=\bar B(s,\ep)\bar B(s,\ep)^\top,\ W^0(s)=\bar B(s,0)\bar B(s,0)^\top={\cal M}(s){\cal M}(s)^\top\
\mbox{ where }\nonumber\\
&\bar B(s,\ep):=[B(s,\ep), \mathbb O_{p(s)\times (n(s)-p(s))}] \mbox{ if } n(s)>p(s),
\bar B(s,\ep):=B(s,\ep) \mbox{ if } n(s)=p(s). \nonumber\eea
Applying Lemma \ref{dur}, we conclude that there exists orthogonal matrix $\Omega\in \mathbb R^{n(s)\times n(s)} $ such
that
$ \bar B(s,0)\Omega={\cal M}(s).$

Let us set ${\cal M}(s,\ep):=\bar B(s,\ep)\Omega.$ Since ${\cal M}(s)>\mathbb O_{p(s)\times n(s)}$ and  ${\cal M}(s,\ep) ={\cal M}(s)+O(\ep)$, we conclude
that   ${\cal M}(s,\ep)> \mathbb O_{p(s)\times n(s)}$
for sufficiently small $\ep\geq 0$. Relations (\ref{11}) follows from the equalities
$W(s,\ep)=\bar B(s,\ep)\bar B(s,\ep)^\top=\bar B(s,\ep)\Omega \Omega^\top \bar B(s,\ep)^\top={\cal M}(s,\ep){\cal M}(s,\ep)^\top.$\end{proof}

\section{System of defining equations for complementarity set }\label{S-4}
The aim of this section is to obtain a system of equations that allows us to describe the complementarity set for the cone $\cop$
in a neighborhood  of a given point $(X^0,U^0)\in \mathbb C(\cop)$.

\begin{theorem}\label{PP3Z} Suppose that for a pair $(X^0,U^0)\in \mathbb C(\cop)$  Assumptions j)--jjj) hold true.
Then for any $(X(\ep),U(\ep))\in \mathbb C(\cop)$ such that
$X(\ep) \to X^0,\ U(\ep)\to U^0 \mbox{ as } \ep \to 0,$
there exist $\ep_0>0$ and
 matrix functions  $ W(s,\ep)\in  {\mathbb S}(p(s)),$ $ s \in S,$ $\ep\in [0,\ep_0],$
such that
\bea& {\cal A}(X(\ep),s)W(s,\ep)+W(s,\ep){\cal A}(X(\ep),s)=\mathbb O_{p(s)\times p(s)}\ \forall\, s\in S,\ \label{lab-0}\\
&U(\ep)=\sum\limits_{s\in S} {\cal B}(W(s,\ep),s).\label{lab-1}\eea
\end{theorem}
\begin{proof}
For $A\subset \mathbb R^p, $ $B\subset \mathbb R^p, $ and $t\in  \mathbb R^p, $ denote
$$\rho(t,B):=\min\limits_{\tau \in B}||t-\tau||,\ \delta(A,B):=\max\limits_{t \in A}\rho(t,B).$$

It is not difficult to show that Assumption jjj)  implies that
\bea& T_a(s,X^0)\cap T_a(\bar s,X^0)=\emptyset \ \, \forall s \in S, \ \forall \bar s \in S,\; s\not=\bar s,\nonumber\\
& \nu_{kj}:= e^\top_{k}X^0\tau(j)>0 \ \forall \,k \in P\setminus P_*(s), \ \forall j \in J(s), \ \forall s \in S.\label{nu}\eea
Hence there exists $\gamma>0$ such that
\be \{t \in T:\rho(t,T_a(s,X^0))\leq \gamma\}\cap \{t \in T:\rho(t,T_a(\bar s,X^0))\leq \gamma\}=\emptyset\ \forall s\in S,\ \forall \bar s\in S,\ s\not=\bar s.\label{01}\ee

Set
\be \delta(\ep):=\delta(T_a(X(\ep)),T_a(X^0)),\ t(\ep):={\rm arg}\,\max\limits_{t \in T_a(X(\ep))}\rho(t, T_a(X^0)).\label{02}\ee
By construction,
\be \delta(\ep)=\rho(t(\ep), T_a(X^0)),\ t(\ep)\in T_a(X(\ep))\subset T.\label{2}\ee
Let $\ep_k, k=1,2,...,$ be a sequence such that $\ep_k\to 0$ as $k\to \infty.$
Let $\delta_*$ and $t^*$ be any limit points of the sequences $\delta(\ep_k),$ $ t(\ep_k)$, $k=1,2,\dots,$ respectively. Notice that such points exist.
Then it follows from (\ref{02}), (\ref{2}) and condition $X(\ep)\to X^0$ as $\ep\to 0 $ that
$\delta_*=\rho(t^*, T_a(X^0))$ and $t^*\in T_a(X^0).$ It follows from these relations that $\delta_*=0$.
Thus we have shown that for any sequence
 $\ep_k, k=1,2,...,$  such that $\ep_k\to 0$ as $k\to \infty$  we have
\bea& \delta(\ep_k)\to 0\ \mbox{ as } k\to \infty.\label{4}\eea
Denote
\bea &  T_a(\ep):=T_a(X(\ep)),\ T_a(s,\ep):=\{t \in  T_a(\ep): \rho(t, T_a(s,X^0))\leq \delta(\ep)\}, \ s\in \bar S,\label{4.1}\\
&\bar S:=\{s \in S: T_a(s,\ep)\not=\emptyset\}.\nonumber\eea
Due to (\ref{01}), (\ref{02}) and (\ref{4}), for sufficiently small $\ep\geq 0$,  we obtain
\be T_a(\ep)=\bigcup\limits_{s \in S}T_a(s,\ep),\ T_a(s,\ep)\cap T_a(\bar s,\ep)=\emptyset \ \forall s \in S,
\ \forall \bar s \in S,\; s\not=\bar s.\label{4.2}\ee

Let us fix $s \in \bar  S$ and consider any vector $\eta(s,\ep)\in T_a(s,\ep).$
By assumption, $X(\ep)\in \cop$ and $\eta(s,\ep)\in T_a(\ep)$, hence
 \be (\eta(s,\ep))^\top X(\ep)\eta(s,\ep)=0\ \Longrightarrow\
 e^\top_k X(\ep)\eta(s,\ep)\left\{\begin{array}{l}
=0\mbox{ if } e^\top_k \eta(s,\ep)>0,\cr
\geq 0\mbox{ if } e^\top_k \eta(s,\ep)=0,\cr\end{array} \right. \forall k \in P.\label{5}\ee
Suppose that there exist $k_0\in P\setminus P_*(s)$ and a sequence  $\ep_k, k=1,2,...,$ such that $\ep_k\to 0$ as $k\to \infty$  and
$e^\top_{k_0}\eta(s,\ep_k) >0\ \forall \, k=1,2,....$ Then due to (\ref{5}) we have
$$e^\top_{k_0}X(\ep_k)\eta(s,\ep_k)=0 \ \forall \, k=1,2,...$$
This implies that
\be e^\top_{k_0}X^0\eta^*(s)=0\label{6}\ee
where $\eta^*(s)$ is a limit point of the sequence $\eta(s,\ep_k),$ $k=1,2,...$
It follows from (\ref{4}), (\ref{4.1}) that $\eta^*(s)\in T_a(s,X^0)$ and, consequently,
$$\eta^*(s)=\sum\limits_{j \in J(s)}\alpha_j\tau(j), \ \alpha_j\geq 0, \, j \in J(s), \, \sum\limits_{j \in J(s)}\alpha_j=1.$$
Taking into account these relations, inequalities (\ref{nu}), and equality (\ref{6}), we obtain
$$0= e^\top_{k_0}X^0\eta^*(s)=\sum\limits_{j \in J(s)}\alpha_je^\top_{k_0}X^0\tau(j)\geq \min\limits_{j \in J(s)}\nu_{k_0 j}>0.$$
We have obtained the contradictory relationships.

Thus we have shown that for sufficiently small $\ep\geq 0$,
  the following equalities hold true
\be e^\top_{k}\eta(s,\ep)=0 \ \forall k \in P\setminus P_*(s), \ \forall \eta(s,\ep)\in T_a(s,\ep), \ \forall s\in \bar S.\label{6.1}\ee

By construction, $(X(\ep), U(\ep))\in \mathbb C(\cop)$, hence matrix $U(\ep)$ admits a presentation
$$ U(\ep)=\sum\limits_{i =1}^{p_*}\alpha_i(\ep)t(i,\ep)(t(i,\ep))^\top,\, \alpha_i(\ep)\geq 0,\, t(i,\ep)\in T_a(\ep),\; i =1,...,p_*, \,
p_*=\frac{p(p+1)}{2}.$$
Then taking into account (\ref{4.1}) (\ref{4.2}), we conclude that  $U(\ep)$ admits a presentation
$U(\ep)=\sum\limits_{s\in S}U(s,\ep)$ where

\centerline{$ U(s,\ep)\in {\cal F}(s,\ep):={\rm cone}\{t\, t^\top, t \in T_a(s,\ep)\}\in \cp,  \; s \in \bar S,\
U(s,\ep)=\mathbb O_{p\times p}, \; s \in S\setminus \bar S. $}

\vspace{2mm}

\noindent It follows from (\ref{6.1}) that by construction, for these matrices, the following equalities hold true:
$$ U_{kq}(s,\ep)=0\ \forall k \in P\setminus P_*(s),\ \forall q\in P, \ \forall s\in S.$$

For $s\in S$, using  matrices $U(s,\ep)\in {\cal CP}(p)$ and  $ U^0(s)\in {\cal CP}(p)$
  let us define   matrices
	$$ W(s,\ep):={\cal A}(U(s,\ep),s)\in {\cal CP}(p(s)),\   W^0(s):={\cal A}(U^0(s),s)\in {\cal CP}(p(s)).$$%\label{lab-5}\ee
 It is evident that relations (\ref{lab-1}) hold true with the matrices $W(s,\ep)$, $s \in S.$

\vspace{2mm}

Now let us show that $\bar S=S$ and  equalities (\ref{lab-0}) hold true. It follows from Assumptions j), jj) and Proposition \ref{P-A}
 that conditions i) and ii) hold true. Hence, for all $s\in S$,  $W(s,\ep)\to W^0(s) \mbox{ as } \ep \to 0$ and
matrix $W^0(s)$ admits a presentation (\ref{W-M}). Then $W(s,\ep)\not=\mathbb O_{p(s)\times p(s)}$ for all $s \in S,$ and consequently
$\bar S=S.$ Moreover, we have shown that
for sufficiently small $\ep_0>0$,
$W(s,\ep)\in {\cal CP}(p(s))\subset {\mathbb S}_+(p(s)) $ $\forall\ \ep\in [0,\ep_0].$  Consequently, it follows from Proposition \ref{P-F} that
for  $s \in S$,  the matrices  $W(s,\ep)$,  $\ep\in [0,\ep_0],$ admit  presentations (\ref{11}).

By assumption, $(X(\ep),U(\ep))\in \mathbb C(\cop)$, hence
\bea& X(s,\ep):={\cal A}(X(\ep),s) \in {\cal COP}(p(s))\  \forall s\in S, \label{lab-10}\\
&X(\ep)\bullet U(\ep)=0  \Longrightarrow\ X(\ep)\bullet U(s,\ep)=0\ \forall s\in S, \Longleftrightarrow\nonumber\\
& X(s,\ep)\bullet W(s,\ep)=0 \ \Longleftrightarrow\
\sum\limits_{i=1}^{n(s)} ({\cal M}_i(s,\ep))^\top  X(s,\ep){\cal M}_i(s,\ep)=0 \ \forall s\in S,\label{13}\eea
where for $s\in S$, ${\cal M}_i(s,\ep)$ is $i$-th column of the matrix ${\cal M}(s,\ep)$.
It follows from the latter equalities and inclusions (\ref{lab-10}) that
\bea&({\cal M}_i(s,\ep))^\top  X(s,\ep){\cal M}_i(s,\ep)=0, \, {\cal M}_i(s,\ep)>0 \ \forall \, i=1,...,n(s) , \, \forall s \in S\Longrightarrow\nonumber\\
& X(s,\ep)\in {\mathbb S}_+(p(s)) \ \forall s\in S.\label{13.1}\eea
Equalities (\ref{lab-0}) follow from (\ref{13}), (\ref{13.1}) and inclusions
$W(s,\ep)\in {\cal CP}(p(s))\subset {\mathbb S}_+(p(s))$ for all $s \in S.$
\end{proof}

\begin{theorem}\label{PP4Z} Suppose that for a pair $(X^0,U^0)\in \mathbb C(\cop)$  Assumptions j)-jjj) hold true.
 Let $X(\ep)\in {\mathbb  S}(p)$ and
$W(s,\ep)\in {\mathbb S}(p(s)),
s\in S,$ be such that
$$X(\ep)\to X^0,\ W(s,\ep)\to W^0(s):={\cal A}(U^0(s),s) \ \mbox{ as } \ep\to 0 \ \forall s\in S$$
where $U^0(s)$, $ s\in S$, are  as in (\ref{1}), and equality (\ref{lab-0}) holds.
Then
for $\ep\in [0,\ep_0],$ where $\ep_0>0$ is a sufficiently small number, the following conditions are fulfilled
\bea& X(\ep) \in \cop,\ W(s,\ep) \in {\cal CP}(p(s))\ \forall s\in S,\label{1zz}\\
&U(\ep):=\sum\limits_{s \in S}{\cal B}(W(s,\ep),s)\in \cp,\ X(\ep)\bullet U(\ep)=0,\label{2zz}\eea
and consequently  $(X(\ep), U(\ep)) \in \mathbb C(\cop)$.
\end{theorem}
\begin{proof} Notice that it follows from Assumptions j)-jj)  and Proposition \ref{P-A}  that  conditions i)-iii) hold true.

 First, let us set $ X(s,\ep)={\cal A}(X(\ep),s)$ and   prove  the following relations
\bea&  X(s,\ep)W(s,\ep)=\mathbb O_{p(s)\times p(s)},\label{zzz}\\
&X(s,\ep)\in {\mathbb S}_+(p(s)),\ W(s,\ep)\in {\mathbb S}_+(p(s)) \ \forall s \in S,\ \forall \ep\in [0,\ep_0].\label{26-01-23-3}\eea
By condition iii), we have
$ X^0(s)+  W^0(s)\in {\rm int}\, {\mathbb  S}_+(p(s)) \ \forall s \in S$
and by assumption
$X(s,\ep)\to X^0(s),\   W(s,\ep)\to   W^0(s) \ \forall s \in S \mbox{ as } \ep \to 0.$
This implies    that
$$ X(s,\ep)+  W(s,\ep)\in {\rm int}\, {\mathbb  S}_+(p(s)) \ \forall s \in S, \ \forall \ep\in [0,\ep_0].$$%\label{22-02-23-1}\ee
It follows from these inclusions, equalities (\ref{lab-0}) and Proposition \ref{P-23-101} that  relations  (\ref{zzz}), (\ref{26-01-23-3})
hold true.

\vspace{2mm}

 It follows from condition ii) that for all $s\in S$, matrix $W^0(s)$ admits a presentation (\ref{W-M}), and
 by assumption  $W(s,\ep)\to W^0(s) \mbox{ as } \ep \to 0
$, and  we have shown that
$W(s,\ep)\in {\mathbb  S}_+(p(s)), $ $\ep\in [0,\ep_0].$  Consequently, it follows from Proposition \ref{P-F} that
for any $s \in S$, the matrices  $W(s,\ep)$,  $\ep\in [0,\ep_0],$ admit  presentations (\ref{11}) and hence
the latter inclusions in (\ref{1zz}) hold true.
This implies that\\

\centerline{$U(\ep):=\sum\limits_{s \in S}{\cal B}(W(s,\ep),s)=\sum\limits_{s\in S}U(s,\ep)\in {\cal CP}(p)
 \mbox{ with } U(s,\ep):={\cal B}(W(s,\ep),s)\ \forall s\in S.$}

Notice that due to specific structure of matrix $U(s,\ep)$ we have that
$X(\ep)\bullet U(s,\ep)=X(s,\ep)\bullet W(s,\ep)$ for all $s \in S.$
Also notice that it was proved above (see (\ref{zzz})) that $X(s,\ep)\bullet W(s,\ep)=0$ for all $s \in S.$
 Taking into account these observations, we conclude that
$X(\ep)\bullet U(\ep)=0$ and relations (\ref{2zz}) are proved.

\vspace{2mm}

 Now let us show that $X(\ep)\in{\cal COP}(p)$ for sufficiently small $\ep\geq 0.$
Denote
$$\tau(\ep):={\rm arg}\{\min \tau^\top X(\ep)\tau, \ \mbox{ s.t. } \tau\in T\}.$$
 Suppose that  there exists a sequence   $\ep_n,$ $n=1,2,...,$ $\ep_n\to 0$ as $n\to \infty,$ such that  $X(\ep_n)\not \in \cop,$ hence
 \be \tau(\ep_n)^\top X(\ep_n)\tau(\ep_n)<0.\label{O100}\ee

 To simplify the presentation, we will denote $\ep_n$ by $\ep$ and suppose that $\ep\to 0.$

 Let $\tau^*$ be a limit point of the sequence $\tau(\ep)$ as $\ep \to 0.$
Hence,  we obtain that
 $$0\leq {\tau^*}^\top X^0\tau^*=\lim\limits_{\ep \to 0}\tau(\ep)^\top X(\ep)\tau(\ep)\leq 0 \ \Longrightarrow\  {\tau^*}^\top X^0\tau^*=0.$$
 This implies that $\tau^*\in T_a(s,X^0)$ for some $s \in S.$ Below we will consider this fixed $s$.

 Let us partition vectors $\tau(\ep)=(\tau_k(\ep), k \in P)^\top$ and $\tau^*=(\tau^*_k, k \in P)^\top$,
and matrices $X(\ep),$ $\Delta X(\ep):= X(\ep)-X^0$ and  $X^0$ w.r.t. the  partition  $P_*(s)\cup(P\setminus P_*(s))$ of the set $P$ as follows
 $$\tau_*(\ep):=(\tau_k(\ep), k \in P_*(s)),\ \tau_0(\ep):=(\tau_k(\ep), k \in P\setminus P_*(s)), \ \tau^*_*:=(\tau_k^*, k \in P_*(s)),$$
\bea&X(\ep)=\left(\begin{array}{cc}
 X_*(\ep)& X_{*0}(\ep)\cr
 (X_{*0}(\ep))^\top& X_0(\ep)
 \end{array}\right),
\Delta  X(\ep)=\left(\begin{array}{cc}
 \Delta X_*(\ep)& \Delta X_{*0}(\ep)\cr
 (\Delta X_{*0}(\ep))^\top& \Delta X_0(\ep)
 \end{array}\right),\nonumber\eea
 \bea &X^0=\left(\begin{array}{cc}
 X^0_*& X^0_{*0}\cr
 (X^0_{*0})^\top& X^0_0
 \end{array}\right).\nonumber\eea
Then
 $$\tau(\ep)=\left(\begin{array}{c}
 \tau_*(\ep)\cr
 \tau_0(\ep)\end{array}\right), \tau^*=\left(\begin{array}{c}
 \tau^*_*\cr
 \mathbf 0\end{array}\right), \ \tau_*(\ep) \to \tau^*_*,\ \tau_0(\ep) \to \mathbf 0 \mbox{ as } \ep \to 0.$$
Notice that here $X_*(\ep)=X(s,\ep)$ with $s$ under consideration.
Using this notation, let us calculate
\bea &\tau(\ep)^\top X(\ep)\tau(\ep)=\left(\begin{array}{c}
\tau_*(\ep)\cr
\mathbf 0
\end{array}\right)^\top X(\ep)\left(\begin{array}{c}
\tau_*(\ep)\cr
\mathbf 0
\end{array}\right)+\nonumber\\
&
2\left(\begin{array}{c}
\tau_*(\ep)\cr
\mathbf 0
\end{array}\right)^\top
X(\ep)\left(\begin{array}{c}
\mathbf 0\cr
\tau_0(\ep)
\end{array}\right)+\left(\begin{array}{c}
\mathbf 0\cr
\tau_0(\ep)
\end{array}\right)^\top X(\ep) \left(\begin{array}{c}
\mathbf 0\cr
\tau_0(\ep)
\end{array}\right)=\nonumber\\
&{\tau_*(\ep)}^\top X_*(\ep)\tau_*(\ep)+2(\tau^*_*+\Delta \tau_*(\ep))^\top (X^0_{*0}+
\Delta X_{*0}(\ep))\tau_0(\ep)+{\tau_0(\ep)}^\top X_0(\ep){\tau_0(\ep)}=\nonumber\\
&{\tau_*(\ep)}^\top X_*(\ep)\tau_*(\ep)+{\tau_0(\ep)}^\top X_0(\ep){\tau_0(\ep)}+\nonumber\\
&2{\tau^*_*}^\top X^0_{*0}\tau_0(\ep)+2{\tau^*_*}^\top\Delta X_{*0}(\ep) \tau_0(\ep)+2\Delta \tau_*(\ep)^\top X^0_{*0}\tau_0(\ep)+
2\Delta \tau_*(\ep)^\top\Delta X_{*0}(\ep) \tau_0(\ep),\nonumber\eea
where $\Delta \tau_*(\ep)=\tau_*(\ep)-\tau^*_*.$

If $ \tau_0(\ep)\equiv 0$ for all  $\ep\to 0, $ then, taking into account that  $X_*(\ep)=X(s,\ep)\in {\mathbb S}_+(p(s))$ (see  (\ref{26-01-23-3})), we have
$\tau(\ep)^\top X(\ep)\tau(\ep)=\tau_*(\ep)^\top X_*(\ep)\tau_*(\ep)\geq 0$.
But this contradicts (\ref{O100}).

Now suppose that $ \tau_0(\ep)\not=0$. Hence, there exists
$$\Delta \tau_0=\lim\limits_{\ep\to 0} \frac{\tau_0(\ep)}{||\tau_0(\ep)||_1}, \ \Delta \tau_0\geq {\bf 0},\
 \sum\limits_{k \in P\setminus P_*(s)}\Delta \tau_{0k}=1.$$
Notice that
$  \tau_0(\ep) \to \mathbf 0,\ \Delta \tau_*(\ep)\to  \mathbf 0,\ \Delta X(\ep)\to \mathbb O_{p\times p}\ \mbox{ as }\ep\to 0.$
Then
$$\lim\limits_{\ep\to 0} \frac{\tau(\ep)^\top X(\ep)\tau(\ep)}{||\tau_0(\ep)||_1}=%\lim\limits_{\ep\to 0} \frac{\alpha(\ep)}{||\tau_0(\ep)||}=$$
\lim\limits_{\ep\to 0} \frac{\tau_*(\ep)^\top X_*(\ep)\tau_*(\ep)}{||\tau_0(\ep)||_1}+2{\tau^*_*}^\top X^0_{*0}\Delta \tau_0.$$
Since  $X_*(\ep)=X(s,\ep)\in {\mathbb S}_+(p(s))$, we have
$\lim\limits_{\ep\to 0} \frac{\tau_*(\ep)^\top X_*(\ep)\tau_*(\ep)}{||\tau_0(\ep)||_1}\geq 0$,  and
$${\tau^*_*}^\top X^0_{*0}\Delta \tau_0=
\sum\limits_{k \in P\setminus P_*(s)}\Delta \tau_{0k}e^\top_kX^0\tau^*\geq \mu(s)\sum\limits_{k \in P\setminus P_*(s)}\Delta \tau_{0k}
=\mu(s),$$
where $\mu(s):=\min\{\nu_{kj}, \; k \in P\setminus P_*(s),\; j \in J(s)\} >0$, $ \nu_{kj}$ are defined in (\ref{nu}).

Hence we obtain that $\tau(\ep)^\top X^0(\ep)\tau(\ep)\geq 0$ for sufficiently small $\ep\geq 0$, but this contradicts (\ref{O100}). Thus, we have proved that
$X^0(\ep)\in \cop$ for sufficiently small $\ep\geq 0.$
 \end{proof}

\begin{theorem}\label{PP3} Suppose that for a pair $(X^0,U^0)\in \mathbb C(\cop)$  Assumptions j)-jjj) hold true.
Then there exists $\ep_0>0$
such that
a pair $( X, U)\in {\mathbb S}(p)\times  {\mathbb S}(p)$, $||( X, U)-(X^0,U^0)||\leq \ep_0$, belongs to the complementarity set
$\mathbb C(\cop)$ iff there exist matrices $ W(s)\in  {\mathbb S}(p(s)),$ $ s \in S,$
such that
\bea & {\cal A}(X,s)W(s)+W(s){\cal A}(X,s)=\mathbb O_{p(s)\times p(s)}\ \forall s\in S,\label{def-eq}\\
 & U=\sum\limits_{s\in S} {\cal B}(W(s),s).\label{def-eq1}\eea
\end{theorem}
{\it {\bf Proof} } follows from Theorems \ref{PP3Z} and \ref{PP4Z}.

\vspace{2mm}

Thus we showed that in a neighborhood  of a given point $(X^0,U^0)\in \mathbb C(\cop)$ satisfying Assumptions j)-jjj),
 the complementarity set $\mathbb C(\cop)$
is uniquely characterized by system of bi-linear (\ref{def-eq}) and linear (\ref{def-eq1}) equations.
It is natural to call this system as a system of defining equations for
cone  $\mathbb C(\cop)$ .

Note that this system of defining equations itself is uniquely defined by
the sets $P_*(s),$ $s \in S,$ which in turn are uniquely determined by the matrix $X^0$.

\vspace{2mm}

Above we supposed that $T_a(X^0)\not=\emptyset.$ Let us analyze the case when $T_a(X^0)=\emptyset$, and consequently $X^0\in {\rm int}\,\cop.$
It is evident that if $X\in {\rm int}\,\cop,$ then condition  $(X,U)\in \mathbb C(\cop)$ implies the equality $U=\mathbb O_{p\times p},$ and
there exists $\ep_0>0$ such that the inequality $||X-X^0||\leq \ep_0$ implies the inclusion $X\in {\rm int}\,\cop.$
Hence, if $T_a(X^0)=\emptyset$, then there exists $\ep_0>0$ such that Theorem \ref{PP3} holds true with $U=U^0=\mathbb O_{p\times p}$ and
all $X\in \cop$ satisfying the inequality $||X-X^0||\leq \ep_0$.

\vspace{2mm}

Let us end this section with a small example illustrating Theorem \ref{PP3}.

Set $ a=(1,-1, 1)^\top$, $b=(1,1,0)^\top,\ c=(0,1 ,1)^\top, $ and
consider a pair $(X^0,U^0)\in {\mathbb C}({\cal COP}(3))$ where
$X^0=\left(\begin{array}{rrr}
0&0&1\cr
0&0&0\cr
1&0&0\end{array}\right)+a\,a^\top\in {\cal COP}(3),\ U^0=bb^\top+cc^\top\in {\cal CP}(3).$

For these matrices, we have $T_a(X^0)=\{\tau(1)=b/2, \tau(2)=c/2\},$ $J=\{1,2\}$, $S=\{1,2\}$,
$J(1)=\{1\},$ $P_*(1)=\{1,2\}= M(1)$,  $J(2)=\{2\},$ $P_*(2)=\{2,3\}=M(2)$, $U^0(1)=bb^\top,$ $U^0(2)=cc^\top.$
One can check that for $(X^0,U^0)$ all Assumptions j)-jjj) are fulfilled.

Using the sets $P_*(1)$ and $P_*(2)$,  let us form the system of defining equations (\ref{def-eq}). For this example the system has the form
\bea& \left(\!\begin{array}{cc}
x_{11}& x_{12}\cr
x_{12}& x_{22}
\end{array}\!\right)\!\!\left(\!\begin{array}{cc}
w_{11}(1)& w_{12}(1)\cr
w_{12}(1)& w_{22}(1)
\end{array}\!\right)\!+\!\left(\!\begin{array}{cc}
w_{11}(1)& w_{12}(1)\cr
w_{12}(1)& w_{22}(1)
\end{array}\!\right)\!\!\left(\!\begin{array}{cc}
x_{11}& x_{12}\cr
x_{12}& x_{22}
\end{array}\!\right)\!=\mathbb O_{2\times 2},\label{sys1}\\
&\left(\!\begin{array}{cc}
 x_{22}& x_{23}\cr
 x_{23}& x_{33}
\end{array}\!\right)\!\!\left(\!\begin{array}{cc}
 w_{22}(2)& w_{23}(2)\cr
 w_{23}(2)& w_{33}(2)
\end{array}\!\right)\!+\!
\left(\!\begin{array}{cc}
 w_{22}(2)& w_{23}(2)\cr
 w_{23}(2)& w_{33}(2)
\end{array}\!\right)\!\!\left(\!\begin{array}{cc}
 x_{22}& x_{23}\cr
 x_{23}& x_{33}
\end{array}\!\right)\!=\mathbb O_{2\times 2}.\label{sys2}\eea
It is a system w.r.t. variables

\vspace{1mm}

$\qquad x_{11},\, x_{12},\,x_{22},\,x_{23},\,x_{33},\,w_{11}(1),\,w_{12}(1),\,w_{22}(1),\,w_{22}(2),\,w_{23}(2),\,w_{33}(2).$

\vspace{2mm}

According to Theorem \ref{PP3},
there exists $\ep_0>0$
such that
a pair $( X, U)\in {\mathbb S}(3)\times  {\mathbb S}(3)$, $||( X, U)-(X^0,U^0)||\leq \ep_0$, belongs to the complementarity set
$\mathbb C({\cal COP}(3))$ iff these matrices $X$ and $U$ have the forms
$$X=\left(\begin{array}{ccc}
x^*_{11}& x^*_{12}& x_{13}\cr
x^*_{12}& x^*_{22}&x^*_{23}\cr
x_{13}&x^*_{23}& x^*_{33}
\end{array}\right),
U=\left(\begin{array}{ccc}
w^*_{11}(1)& w^*_{12}(1)& 0\cr
w^*_{12}(1)& w^*_{22}(1)+w^*_{22}(2)&w^*_{23}(2)\cr
0&w^*_{23}(2)& w^*_{33}(2)
\end{array}\right)$$
where
$(x^*_{11},\, x^*_{12},\,x^*_{22},\,x^*_{23},\,x^*_{33},\,w^*_{11}(1),\,w^*_{12}(1),\,w^*_{22}(1),\,w^*_{22}(2),\,w^*_{23}(2),\,w^*_{33}(2))$
is a solution to system (\ref{sys1}), (\ref{sys2}).

\section{Independence of bi-linear functions forming the system of defining equations}\label{S-6}
In Section \ref{S-4}, we obtained a system of defining equations (\ref{def-eq}), (\ref{def-eq1}), which allows one to uniquely characterize the complementarity set $\mathbb C(\cop)$
in a neighborhood  of a given point $(X^0,U^0)\in \mathbb C(\cop)$ satisfying Assumptions j)-jjj).
In this system, the main role play equations (\ref{def-eq}) since equations (\ref{def-eq1}) are linear and  needed only for defining matrix $U$ using
matrices $W(s)$, $s \in S,$ satisfying  (\ref{def-eq}). In this section, we show that Jacobian matrix of the system of equations (\ref{def-eq}), calculated at
point $(X^0,U^0)\in \mathbb C(\cop)$, has full row rank.

Given  a pair $(X^0,U^0)\in \mathbb C(\cop)$ and a presentation of $U^0$ in the form (\ref{1}),
let us denote $p_*(s)=p(s)(p(s))+1)/2, s \in S,$ $p_*=p(p+1)/2,$ $m=\sum\limits_{s\in S}p_*(s).$
Using  $X\in {\mathbb S}(p)$ and $W(s)\in  {\mathbb S}(p(s)),$ $ s \in S$, as parameters, let us form  vectors and {\bf bi-linear}  vector-functions
\be z=\begin{pmatrix}
{\rm svec}(X)\cr
{\rm svec}(W(s))\cr
s \in S\end{pmatrix}\in \mathbb R^{p_*+m},\  z^0=\begin{pmatrix}
{\rm svec}(X^0)\cr
{\rm svec}({\cal A}(U^0(s),s))\cr
s \in S\end{pmatrix}\in \mathbb R^{p_*+m},\label{zz}\ee
\bea &\Omega(z,s):={\rm svec}({\cal A}(X,s)W(s)+W(s){\cal A}(X,s))\in \mathbb R^{p_*(s)}, s \in S,\nonumber\\
&\Omega(z):=(\Omega^\top (z,s), s \in S)^\top\in \mathbb R^m.\nonumber\eea

It is easy to see that system (\ref{def-eq}) can be rewritten as
$\Omega(z)={\bf 0}.$

\begin{proposition}\label{PP10} Suppose that for a pair $(X^0,U^0)\in \mathbb C(\cop)$ Assumption jj)
and condition iii) from Proposition \ref{P-A}  are satisfied.
Let $U^0(s),s \in S,$ be matrices satisfying  (\ref{1}), vector $z^0$ and function $\Omega(z)$ be as defined above.
Then the matrix $\frac{\partial\,\Omega(z^0)}{\partial z}\in \mathbb R^{m\times(p_*+m)}$ has full rank:
${\rm rank}\frac{\partial\,\Omega(z^0)}{\partial z}\in \mathbb R^{m\times(p_*+m)}=m.$
\end{proposition}
\begin{proof} For given $M,N\in \mathbb R^{p\times p}$, denote by $M\otimes_s
 N$ a symmetric Kronecker product of these matrices,
see \cite{alizadeh1998primal} for exact definition and properties.
Notice that $M\otimes_s N$ is $p_*\times p_*$ matrix, $M\otimes_s N=N\otimes_s M$,  it is symmetric if $M$ and $N$ are symmetric,
 and for $ X,\,U\, \in {\mathbb S}(p)$ we have
\be (X\otimes_s E(p)){\rm svec}(U)=\frac{1}{2}{\rm svec}(XU+UX),\label{kr}\ee
where $E(p)$ is identity $p\times p$ matrix.
Denote
$$X^0(s):={\cal A}(X^0,s)\in {\mathbb  S}(p(s)),\  W^0(s):={\cal A}(U^0(s),s)\in{\mathbb S}(p(s)), \ s\in S,$$
$${\cal K}(s):= W^0(s)\otimes_s E(p(s))\in {\mathbb S}(p_*(s)),\ {\cal L}(s):= X^0(s)\otimes_s E(p(s))\in {\mathbb S}(p_*(s)),\ \ s\in S.$$
Let ${\cal K}_{ij}(s)\in \mathbb R^{p_*(s)},$ $(i,j)\in V(P_*(s)),$ be columns of the matrix ${\cal K}(s)$ for $s \in S, $ and let $S=\{1,...,s_*\}.$
Then the matrix $\frac{\partial\,\Omega(z^0)}{\partial z}$ has the form
$$\frac{\partial\,\Omega(z^0)}{\partial z}=
\begin{pmatrix}
{\cal Q}(1) & {\cal L}(1)& \mathbb O&\mathbb O&...&\mathbb O\cr
{\cal Q}(2) & \mathbb O & {\cal L}(2)&\mathbb O&...&\mathbb O\cr
...&...&...&...&...&...\cr
{\cal Q}(s_*) & \mathbb O &\mathbb O&\mathbb O&...& {\cal L}(s_*)\end{pmatrix}$$
where for $s \in S,$ ${\cal Q}(s)$ is $p_*(s)\times p_*$ matrix with columns ${\cal Q}_{ij}(s)$, $(i,j)\in V(P):$
$${\cal Q}_{ij}(s)={\bf 0}\ \forall (i,j)\in V(P)\setminus V(P_*(s)),\ {\cal Q}_{ij}(s)=
{\cal K}_{ij}(s) \ \forall (i,j)\in V(P_*(s)).$$

Suppose that ${\rm rank}\frac{\partial\,\Omega(z^0)}{\partial z}\in \mathbb R^{m\times(p_*+m)}<m.$ Then there exists a set of matrices
$Y(s)\in {\mathbb S}(p(s)), \ s \in S,$
such that $\sum\limits_{s\in S}||Y(s)||>0$ and
$\begin{pmatrix}
{\rm svec}(Y(s))\cr
s\in S\end{pmatrix}^\top \frac{\partial\,\Omega(z^0)}{\partial z}={\bf 0}^\top.$ The latter equality implies the equalities
$$\sum\limits_{s\in S}({\rm svec}(Y(s)))^\top {\cal Q}(s)={\bf 0}^\top,\ ({\rm svec}(Y(s)))^\top {\cal L}(s)={\bf 0}^\top \ \forall s\in S.$$
Then it follows from these equalities and (\ref{kr}) that
$$\sum\limits_{s\in S}{\cal B}(Z(s),s)=\mathbb O_{p\times p}, \ X^0(s)Y(s)+Y(s)X^0(s)=\mathbb O_{p(s)\times p(s)}\ \forall s \in S,$$
where $Z(s)=W^0(s)Y(s)+Y(s)W^0(s)$ for all $s \in S.$
It follows from statement a) in Proposition \ref{PPA1}  that $X^0(s)Z(s)= \mathbb O_{p(s)\times p(s)}\ \forall s \in S.$

By construction, for all $s \in S$, $X^0(s)\in {\mathbb S}_+(p(s)),$ $X^0(s)t(s)={\bf 0}$ where
$t(s)=\sum\limits_{j \in J(s)} \tau_*(j)>{\bf 0}$, $\tau_*(j)=(e^\top_k\tau(j), k \in P_*(s))^\top, $ $ j \in J(s).$
Hence it follows from Proposition \ref{P01} that the vectors $\tau_*(j), j \in J_b(s),$ form a basis of the subspace
$\{t\in \mathbb R^{p(s)}:X^0(s)t={\bf 0}\}. $  Then applying Proposition \ref{PPA2} we obtain that matrix $Z(s)$ can be written in the
form
$$Z(s)=\sum\limits_{(i,j)\in V(J_b(s))}\beta_{ij}(s)(\tau_*(i)+\tau_*(j))(\tau_*(i)+\tau_*(j))^\top.$$
This implies that ${\cal B}(Z(s),s)=\sum\limits_{(i,j)\in V(J_b(s))}\beta_{ij}(s)(\tau(i)+\tau(j))(\tau(i)+\tau(j))^\top$
and hence
$$\mathbb O_{p\times p}=\sum\limits_{s\in S}{\cal B}(Z(s),s)=
\sum\limits_{s\in S}\sum\limits_{(i,j)\in V(J_b(s))}\beta_{ij}(s)(\tau(i)+\tau(j))(\tau(i)+\tau(j))^\top.$$
 It follows from this equality and Assumption jj) that
$\beta_{ij}(s)=0$ for all $(i,j)\in V(J_b(s))$ and all $s \in S,$ and consequently $Z(s)=\mathbb O_{p(s)\times p(s)}\ \forall s \in S$. Taking into account these
equalities, condition iii) and statement b) in
  Proposition \ref{PPA1}, we obtain  the equalities $Y(s)=\mathbb O_{p(s)\times p(s)}\ \forall s \in S$. But this contradicts the condition
 $\sum\limits_{s\in S}||Y(s)||>0$.
\end{proof}

Let   $\Omega_i(z),$ $ i=1,...,m,$ be bi-linear functions forming vector-function $\Omega(z)$ from the system of defining equations.
It follows from proposition proved above that the functions $\Omega_i(z),$ $ i=1,...,m,$ are independent in
 a neighborhood  of  a point $z^0$ constructed on the base of a point $(X^0,U^0)\in \mathbb C(\cop)$ by rules (\ref{zz}).

\vspace{2mm}
Let us consider example from Section \ref{S-4}. In  this example, we have $p=3,$ $p_*=6,$  $p(s)=2,$ $p_*(s)=3,$ for $s\in S=\{1,2\},$ and
system (\ref{def-eq}) has the form (\ref{sys1}), (\ref{sys2}).
 It is a system
of $m:=p_*(1)+p_*(2)=6$ bi-linear equations w.r.t. to vector
$$z=(x_{11}, x_{12},x_{13},x_{22},x_{23},x_{33},w_{11}(1),w_{12}(1),w_{22}(1),w_{22}(2),w_{23}(2),w_{33}(2))^\top\in \mathbb R^{p_*+m}.$$
One can check that, as it is  stated in Proposition \ref{PP10},   the bi-linear functions forming this system are  independent in
 a neighborhood  of  a point $z^0$ constructed on the base of $X^0,$ $ U^0(1)$, $U^0(2)$ by rules (\ref{zz}).

\section{Some comments on Assumptions j)-jjj)}\label{S-5}
In this section, we will show that  none of the assumptions j)-jjj) can be omitted in the formulations of the theorems
Theorems \ref{PP3Z},  \ref{PP4Z} and Proposition \ref{PP10}.
 
Let us start with Assumption jj).
One can show that under Assumption jj), for a given  $U^0\in {\cp}$ such that $X^0\bullet U^0=0$  there is a
unique presentation in the form (\ref{1}).
Note that this assumption does not exclude the situation when for some $s\in S,$ matrix ${U}^0(s)$ from unique presentation in the form (\ref{1})
may have several different cp factorizations.

If Assumption j) is fulfilled but
 Assumption jj) is not satisfied, then  there  exist several sets of matrices, for example,
${\mathbb U}^0=\{U^0(s), s \in S\}$ and $\bar {\mathbb U}=\{\bar U(s), s \in S\}$ such that ${\mathbb U}^0\not=\bar {\mathbb U}$ and
$U^0=\sum\limits_{s\in S}{ U}^0(s)=\sum\limits_{s\in S}\bar { U}(s), \ U^0(s),\, \bar U(s)\in {\cal F}(s), \ \forall s\in S.$
Notice that now we can not guarantee the fulfillment of the condition ii) for matrices from both sets ${\mathbb U}^0$ and $\bar {\mathbb U}$.
 But as it will be shown below the fulfillment of this condition is essential.

Also note that if Assumption jj) is violated,  then ${\rm rank}\frac{\partial\,\Omega(z^0)}{\partial z}<m$ and as a result statement
 of Proposition \ref{PP10} does not hold true.

\vspace{3mm}

{\it Violation of Assumption j) in  Theorem \ref{PP3Z}}.
Set $p=5$ and  for parameter vector $\theta=(\theta_j, j=1,...,5)$
define a matrix
{\small $$H(\theta)=\left(\begin{array}{ccccc}
1 &-\cos(\theta_4) & \cos(\theta_4+\theta_5)& \cos(\theta_2+\theta_3) & -\cos(\theta_3)\cr
    -\cos(\theta_4)&  1& -\cos(\theta_5)&  \cos(\theta_1+\theta_5)& \cos(\theta_4+\theta_3) \cr
     \cos(\theta_4+\theta_5) &-\cos(\theta_5) &1 &-\cos(\theta_1) & \cos(\theta_1+\theta_2)\cr
     \cos(\theta_3+\theta_2)& \cos(\theta_1+\theta_5) &-\cos(\theta_1)& 1 & -\cos(\theta_2)\cr
     -\cos(\theta_3)& \cos(\theta_3+\theta_4) &\cos(\theta_1+\theta_2)& -\cos(\theta_2) &1\end{array} \right)$$ }
		and vectors
		\bea &a(\theta) =(\cos(\theta_4 + \theta_5),\,
-\cos (\theta_5),\,
1,\,
-\cos (\theta_1),\,
\cos(\theta_1 + \theta_2))^\top,\nonumber\\
 &b(\theta) =
(\sin(\theta_4 + \theta_5),\,
-\sin (\theta_5),\,
0,\,
\sin (\theta_1),\,
-\sin(\theta_1 + \theta_2))^\top.\nonumber\eea
Consider  a vector $\theta^*=(\theta^*_j, j=1,...,5)$
such that $\theta^*_j>0,$ $ j=1,...,5,$ $\sum\limits_{j=1}^5\theta^*_j=\pi.$
 Let us  fix this vector and set $X^0=H(\theta^*).$

It is known (see \cite{zhang2018completely}) that $H(\theta^*)=a(\theta^*)(a(\theta^*))^\top+b(\theta^*)(b(\theta^*))^\top$,
$X^0\in {\mathbb S}_+(5)\subset {\cal COP}(5),$ $T_a(X^0)={\rm conv} \{\tau(j,\theta^*),j=1,...,5\}$ where
$$ \tau(1,\theta)=\left(\begin{array}{c}
\sin(\theta_5)\cr
   \sin(\theta_4+\theta_5)  \cr
     \sin(\theta_4)  \cr
         0 \cr
         0 \end{array} \right),
				\tau(2,\theta)=\left(\begin{array}{c}
				   0\cr
    \sin(\theta_1)\cr
      \sin(\theta_1+\theta_5)\cr
          \sin(\theta_5)\cr
           0  \end{array} \right),
		\tau(3,\theta)=\left(\begin{array}{c} 			
		    0   \cr
      0    \cr
      \sin(\theta_2)  \cr
          \sin(\theta_1+\theta_2) \cr
           \sin(\theta_1)  \end{array} \right),$$
				$$\tau(4,\theta)=\left(\begin{array}{c} 				
  \sin(\theta_2)  \cr
        0  \cr
       0   \cr
          \sin(\theta_3) \cr
         \sin(\theta_3+\theta_2)\end{array} \right),
\tau(5,\theta)=\left(\begin{array}{c}
   \sin(\theta_4+\theta_3)\cr
      \sin(\theta_3)\cr
          0\cr
            0 \cr
          \sin(\theta_4)\end{array} \right).$$
					Hence $S=\{1\},$ $P_*(1)=\{1,...,5\},$ and it is evident that Assumptions jj) and jjj)  are satisfied.
								
					Let us set $U^0=U(\theta^*)$ where $U(\theta): =\sum\limits_{j=1}^5\tau(j,\theta)(\tau(j,\theta))^\top.$
					By construction
$U(\theta^*)\in {\cal CP}(5)$ and $H(\theta^*)\bullet U(\theta^*)=0,$ consequently $(X^0, U^0)\in \mathbb C({\cal COP}(5)).$
One can check that  $U^0$ is not strictly complementary to $X^0$, i.e., Assumption j)  does not hold true.

For $0<\ep<\theta^*_1$, let us set $\theta(\ep)=(\theta^*_1-\ep,\theta^*_j,j=2,...,5)$ and consider matrices
$$X(\ep):=H(\theta(\ep)), \ U(\ep):=U(\theta(\ep)).$$
It is evident that $X(\ep)\to X^0$ and $U(\ep)\to U^0.$ It is known (see \cite{zhang2018completely}) that $X(\ep)\in {\cal COP}(p)$, $U(\ep)\in {\cal CP}(p)$
and $X(\ep)\bullet U(\ep)=0.$
It is also known that $X(\ep)\not \in {\mathbb S}_+(5). $

Let us show that
$X(\ep)U(\ep)+U(\ep)X(\ep)\not=\mathbb O_{5\times 5}.$ Do to this, first let us show that
$X^0+U^0\in {\rm int}\, {\mathbb S}_+(5).$

Remind that by construction $X^0=a(\theta^*)(a(\theta^*))^\top+b(\theta^*)(b(\theta^*))^\top$.
Since $X^0\in {\mathbb S}_+(5)$ and $U^0\in {\mathbb S}_+(5)$, hence it is evident that $X^0+U^0\in {\mathbb S}_+(5).$
Let $t\in \mathbb R^5$ be such that
\bea &t^\top(X^0+U^0)t=0\ \Longrightarrow\ t^\top X^0t=0,  \ t^\top U^0t=0, \ \Longrightarrow\nonumber\\
&t^\top a(\theta^*)=0,\, t^\top b(\theta^*) ,\, t^\top \tau(j,\theta^*)=0,\, j=1,...,5.\nonumber\eea
It was shown in \cite{zhang2018completely} that this system has only trivial solution $t={\bf 0}$. Thus we have shown that $X^0+U^0\in{\rm int } {\mathbb S}_+(5).$
Consequently
$$X(\ep)+U(\ep)\in {\rm int }\, {\mathbb S}_+(5) \mbox{ for sufficiently small } \ep>0.$$

Now let us suppose that for sufficiently small $\ep>0$, the equality
$X(\ep)U(\ep)+U(\ep)X(\ep)=\mathbb O_{5\times 5}$  holds true. Then it follows from Proposition \ref{P-23-101} that $X(\ep)\in {\mathbb S}_+(p)$,
 but this contradicts the known condition $X(\ep)\not \in {\mathbb S}_+(p)$.

Thus we have shown that all conditions of Theorem \ref{PP3Z}, except Assumption j), are fulfilled, but   statements of this theorem do not hold true.
Hence the assumption is essential in  this theorem.

\vspace{1mm}

{\it Violation of assumption jjj) in  Theorem \ref{PP3Z}}. Let us set $p=3$ and consider
\bea& X^0=\left(\begin{array}{rrr}
1 & -1 &0\cr
-1&1&0\cr
0&0&1\end{array}\right)=a\,a^\top+b\, b^\top,
U^0=\begin{pmatrix}
1 & 1 &0\cr
1&1&0\cr
0&0&0\end{pmatrix}=\tau\tau^\top,\nonumber\\
&a^\top=(1,-1,0),\ b^\top=(0,0,1),\ \tau^\top=(1,1,0).\nonumber\eea

It is easy to see that $(X^0,U^0)\in \mathbb C(\cop),$ $T_a(X^0)=\tau=T_a(1,X^0).$ Hence $J=\{1\},$
$\tau(1)=\tau,$   $S=\{1\},$ $J(1)=J,$ $ P_*(1)=\{1,2\},$ $M(j)=\{1,2,3\},j \in J,$
$ U^0(1)=U^0,$
$X^0(1)=\begin{pmatrix}
1&-1\cr
-1&1\end{pmatrix},$ $W^0(1)=\begin{pmatrix}
1&1\cr
1&1\end{pmatrix}.$
Assumption j) and jj) are fulfiled, Assumption jjj) is not fulfilled.

Set
\bea& a^\top(\ep)=(1,-1,\ep),\ b^\top(\ep)=(0,-\ep,1),\ \tau^\top(\ep)=(1-\ep^2,1,\ep),\nonumber\\
&X(\ep)=a(\ep)\,a^\top(\ep)+b(\ep)\, b^\top(\ep),\ U(\ep)=\tau(\ep) \tau^\top(\ep).\nonumber\eea
Hence
$$X(1,\ep)=\begin{pmatrix}
1& -1\cr
-1& 1+\ep^2\end{pmatrix}, W(1,\ep)=\begin{pmatrix}
(1-\ep^2)^2& (1-\ep^2)\cr
(1-\ep^2)& 1\end{pmatrix}.$$
By construction
$X(\ep) \to X^0, \ U(\ep)\to U^0 \mbox{ as } \ep\to 0$, $X(\ep)\bullet U(\ep)=0$ and
$X(\ep)\in {\mathbb S}_+(3)\subset {\cal COP}(3),$ $  U(\ep)\in {\cal CP}(3)$ for all $\ep\in [0,1].$

 Thus we have shown that all conditions of Theorem \ref{PP3Z}  except Assumption jjj) are satisfied.
But the statements of the theorem  is not true since
$$X(1,\ep)W(1,\ep)+U(1,\ep)X(1,\ep) \not =\mathbb O_{2\times 2} \mbox{  and } U_{3q}(\ep) \not=0 \ \forall q=1,2,3.$$
Note that in this example, violation of Assumption jjj) implies that $X^0$ is not strictly complementary to $U^0.$

\vspace{2mm}

{\it Violation of Assumption j) in Theorem \ref{PP4Z}.} Let us set $p=4$,
$$X^0=\begin{pmatrix}
1 & 1 &1&1\cr
1&0&0&0\cr
1&0&0&0\cr
1&0&0&0\end{pmatrix},
U^0=\begin{pmatrix}
0&0 & 0 &0\cr
0&1&1&1\cr
0&1&1&1\cr
0&1&1&1\end{pmatrix}=(e_2+e_3+e_4)(e_2+e_3+e_4)^\top.$$
It is easy to see that $(X^0,U^0)\in \mathbb C(\cop),$ $T_a(X^0)={\rm conv}\{e_2,e_3,e_4\}=T_a(1,X^0).$ Hence $J=\{1,2,3\},$
$\tau(1)=e_2,$ $\tau(2)=e_3,$  $\tau(3)=e_4,$ $S=\{1\},$ $J(1)=J,$ $ P_*(1)=\{2,3,4\},$ $M(j)=\{2,3,4\},j \in J,$
$ U^0(1)=U^0,$
$X^0(1)=\mathbb O_{3\times 3},$
$W^0(1)
=(1,1,1) (1,1,1)^\top.$
Hence we see that  Assumptions  jj) and jjj) are fulfilled but Assumption j) is not fulfilled. Notice that in this example
despite the violation of Assumption j), conditions i) and ii) are satisfied.

Consider
$$X(\ep)=\begin{pmatrix}
1 & 1 &1&1\cr
1 &0&-\ep&\ep\cr
1&-\ep&0&\ep\cr
1&\ep&\ep&-2\ep\end{pmatrix},\,W(1,\ep)=\begin{pmatrix}
1&1&1\cr
1&1&1\cr
1&1&1\end{pmatrix} \ \Longrightarrow \
 X(1,\ep)=\begin{pmatrix}
0&-\ep&\ep\cr
-\ep&0&\ep\cr
\ep&\ep&-2\ep\end{pmatrix}.$$
It is easy to see that
$X(\ep) \to X^0, \ W(1,\ep) \to W^0(1) \mbox{ as } \ep\to 0$ and $ X(1,\ep)W(1,\ep)+U(1,\ep)X(1,\ep)=\mathbb O_{3\times 3} \; \forall \ep\in \mathbb R.$

 Thus  all conditions of Theorem \ref{PP4Z}  except Assumption j) are satisfied.
But the statements of the theorem  is not true since $X(\ep)\not\in {\cal COP}(4)$ for $\ep>0.$

\vspace{3mm}

{\it Violation of Assumption jjj) in Theorem \ref{PP4Z}.} Let us $p=3$ and consider
$$X^0=\begin{pmatrix}
1 & 0 &0\cr
0&0&0\cr
0&0&0\end{pmatrix}, U^0=\begin{pmatrix}
0 & 0 &0\cr
0&2&1\cr
0&1&2\end{pmatrix}=e_2e^\top_2+e_3e^\top_3+(e_2+e_3)(e_2+e_3)^\top.$$
It is easy to see that $(X^0,U^0)\in \mathbb C(\cop),$ $T_a(X^0)={\rm conv}\{e_2,e_3\}=T_a(1,X^0).$ Hence $J=\{1,2\},$
$\tau(1)=e_2,$ $\tau(2)=e_3,$  $S=\{1\},$ $J(1)=J,$ $ P_*(1)=\{2,3\},$ $M(j)=\{1,2,3\},j \in J,$
$ U^0(1)=U^0,$
$X^0(1)=\begin{pmatrix}
0&0\cr
0&0\end{pmatrix},$ $W^0(1)=\begin{pmatrix}
2&1\cr
1&2\end{pmatrix}.$
By construction (see Lemma \ref{L-1}), $U^0$ is strictly complementary to $X^0$, hence Assumption j) is fulfilled.
 Assumption jj) is also fulfilled as $S=\{1\}.$
Since $P_*(1)\not=M(j), j \in J,$  we conclude that Assumption jjj) is not fulfilled.

Consider
$X(\ep)=\begin{pmatrix}
1 & -\ep &0\cr
-\ep&0&0\cr
0&0&0\end{pmatrix},\ W(1,\ep)=\begin{pmatrix}
2&1\cr
1&2\end{pmatrix} \ \Longrightarrow\ X(1,\ep)=\begin{pmatrix}
0&0\cr
0&0\end{pmatrix}.$

It is easy to see that
\bea &X(\ep) \to X^0, \, W(1,\ep) \to W^0(1) \mbox{ as } \ep\to 0;\label{limm}\\
& X(1,\ep)W(1,\ep)+U(1,\ep)X(1,\ep)=\mathbb O_{2\times 2} \
\forall \ep\in \mathbb R.\nonumber\eea
 Thus we have shown that all conditions of Theorem \ref{PP4Z}  except Assumption jjj) are satisfied.
But the statements of the theorem  is not true since $X(\ep) \not \in {\cal COP}(3)$ for $\ep>0.$

Note that in this example, violation of Assumption jjj) leads to the fact  that $X^0$ is not strictly complementary to $U^0.$

\vspace{2mm}

 {\it Violation of condition ii) in Theorem \ref{PP4Z}.} Notice that  Theorem \ref{PP4Z} is formulated under Assumptions j)-jjj). But in fact we use
 only Assumption jjj) and conditions i)-iii) which are consequence of  Assumptions j) and jj).  In an example below, Assumption jjj) and conditions
i) and iii) are
fulfilled, but  the only condition ii) is violated. As a result statements of  Theorem \ref{PP4Z} do not hold true.

Let us set $p=3$ and consider
$$X^0=\begin{pmatrix}
1 & 1 &1\cr
1&0&0\cr
1&0&0\end{pmatrix}, U^0=\begin{pmatrix}
0 & 0 &0\cr
0&1&0\cr
0&0&1\end{pmatrix}.$$
It is easy to see that $(X^0,U^0)\in \mathbb C(\cop),$ $T_a(X^0)={\rm conv}\{e_2,e_3\}=T_a(1,X^0).$ Hence $J=\{1,2\},$
$\tau(1)=e_2,$ $\tau(2)=e_3,$  $S=\{1\},$ $J(1)=J,$ $ P_*(1)=\{2,3\},$ $M(j)=\{2,3\},j \in J,$
$ U^0(1)=U^0,$
$X^0(1)=\begin{pmatrix}
0&0\cr
0&0\end{pmatrix},$ $W^0(1)=\begin{pmatrix}
1&0\cr
0&1\end{pmatrix}.$
Since  ${\rm rank}(X^0(1)+W^0(1))=2$ , we have  $(X^0(1)+W^0(1))\in {\rm int}\,{\mathbb S}_+(2).$ Thus  in this example
Assumptions jj), jjj) and conditions i) and iii) are  fulfilled
 but condition ii) is violated.

Let us set
$$X(\ep)=X^0\ \Longrightarrow\ X^0(1)= X(1,\ep)=\begin{pmatrix}
0&0\cr
0&0\end{pmatrix} ,\ W(1,\ep)=\begin{pmatrix}
1&-\ep\cr
-\ep&1\end{pmatrix}. $$
Then by construction, relations (\ref{limm}) hold true.

Thus we have shown that all conditions (under which Theorem \ref{PP4Z} is proved)  except condition ii) are satisfied.
But the statements of the theorem  is not true. In fact, in the example we have  $W(1,\ep) \not \in {\cal CP}(2)$ for $\ep>0 $
and as result  $U(\ep):=\sum\limits_{s \in S}{\cal B}(W(s,\ep),s)={\cal B}(W(1,\ep),1)\not \in {\cal CP}(p).$

In this section,  we have shown that none of the assumptions j)-jjj) can be omitted in the formulations of the theorems
Theorems \ref{PP3Z},  \ref{PP4Z} and Proposition \ref{PP10} and none of the assumptions is a consequence of other assumptions.
 However, we guess that in Theorems \ref{PP3Z},  \ref{PP4Z} the assumptions jj), jjj) can be replaced by some less restrictive ones.

\section{Conclusion}
The main motivation for these studies is our desire to explore parametric conic problems over the cones of copositive and completely positive matrices.
 This led us to examine for these cones  the complementarity relations,
which play a central role in the study of optimization problem.  
Based on the analysis of the properties of the elements of the complementarity set under consideration, we obtained
a system of bi-linear and linear equations w.r.t. extended set of variables, which allow us to completely
describe the complementarity set in a neighborhood of a point  from this set under some non-degeneracy assumptions.
This result will be the basis for our further research related to the sensitivity of the solution to copositive programming problems.

Also, these results can serve as the basis for further research of  the properties of the complementarity set,
 which will allow us to relax the assumptions made in the paper.

\bibliographystyle{tfnlm}
\bibliography{interactnOLGA}

\appendix
\section{Proofs of auxiliary statements}

\begin{proposition} \label{P-23-101} Suppose that $U\in {\mathbb S}(p)$,  $X\in {\mathbb  S}(p)$, $UX+XU=\mathbb O_{p\times p}$ and $X+U\in {\rm int }\,{\mathbb S}_+(p).$
Then $X\in{\mathbb  S}_+(p),$ $U\in {\mathbb  S}_+(p)$ and $UX=\mathbb O_{p\times p}$.
\end{proposition}
\begin{proof} If $U\in {\mathbb  S}(p)$ and $X\in {\mathbb  S}(p)$, then there exist matrices $Q$ and $G$ such that
\bea &U=Q\Lambda Q^\top,\ X=GMG^\top ,\ Q^\top Q=E(p), \ G^\top G=E(p) \ \mbox{ where }\nonumber\\
&\Lambda={\rm diag}(\lambda_k,k=1,...,p), \ M={\rm diag}(\mu_k,k=1,...,p),\nonumber\eea
 $\lambda_k,k=1,...,p$, are eigenvalues of matrix $U$, $\mu_k,k=1,...,p$, are eigenvalues of matrix $X$.
 Then conditions $UX+XU=\mathbb O_{p\times p}$ and $X+U\in {\rm int }\,{\mathbb S}_+(p)$ take the form
$$\mathbb O_{p\times p}=Q\Lambda Q^\top GMG^\top   +GMG^\top Q\Lambda Q^\top, \
Q\Lambda Q^\top+ GMG^\top\in {\rm int} \, {\mathbb  S}_+(p).$$
It follows from these relations and condition ${\rm det} Q\not=0$ that
$$ \Lambda F+F\Lambda=\mathbb O, \ \ \Lambda +F\in {\rm int} \, {\mathbb  S}_+(p) \ \mbox{ where } F:=Q^\top GMG^\top Q.$$
This implies that
$ 2\lambda_kF_{kk}=0,\   \lambda_k+F_{kk}>0\ \forall \, k=1,...,p,$ wherefrom we obtain that $\lambda_k\geq 0,$ $F_{kk}\geq 0$ for $k=1,...,p.$

If $\lambda_k\geq 0$ for all $k=1,...,p,$ then it is evident that
$U=Q\Lambda Q^\top \in {\mathbb  S}_+(p).$

Notice that $F_{kk}=\mu_k||\omega_k||^2,$ where $\omega_k=G^\top Q e_k.$ If  ${\rm det}\, Q\not=0$ and ${\rm det}\, G\not=0$, then
$||\omega_k||^2>0$  for all $k=1,...,p.$ These inequalities and relations $F_{kk}\geq 0$, $F_{kk}=\mu_k||\omega_k||^2,$ $k=1,...,p,$ imply the inequalities
$\mu_k\geq 0$  for all $k=1,...,p$  and, hence, $X=GM G^\top \in {\mathbb  S}_+(p).$
 Thus we have proved that $U \in {\mathbb S}_+(p)$ and $X\in {\mathbb  S}_+(p).$

It follows from the condition $UX+XU=\mathbb O_{p\times p}$ that
$0={\rm trace}(UX+XU)=2{\rm trace}(UX).$
 It is easy to show that the relations
$U \in {\mathbb  S}_+(p)$, $X\in {\mathbb  S}_+(p)$ and ${\rm trace}(UX)=0$
imply the equality $UX=\mathbb O_{p\times p}$.\end{proof}

\begin{proposition}\label{PPA1} Let $X\in {\mathbb S}_+(p),$ $W\in {\mathbb S}_+(p),$ and $Y\in {\mathbb S}(p)$ be such that
$WX=\mathbb O_{p\times p},$
\be XY+YX=\mathbb O_{p\times p}.\label{a-1}\ee
Denote $Z:=YW+WY\in {\mathbb S}(p)$. Then
a) $ XZ=\mathbb O_{p\times p}$, b)   the equality $Z=\mathbb O_{p\times p}$ implies the equality
$Y=\mathbb O_{p\times p}$ if $(X+W)\in {\rm int}\,{\mathbb S}_+(p).$
\end{proposition}
\begin{proof} It follows from inclusions $X\in {\mathbb S}_+(p),$ $W\in {\mathbb S}_+(p),$ and equality
$WX=\mathbb O_{p\times p}$ that there exist a matrix $Q\in \mathbb R^{p\times p}$ and numbers $w_k\geq 0,$ $x_k\geq 0,$ $k=1,...,p,$
such that $Q^\top Q=E(p)$ and
$$QXQ^\top={\rm diag}(x_k,k=1,...,p), \ QWQ^\top={\rm diag}(w_k,k=1,...,p),  \ w_kx_k=0\ \forall \, k=1,...,p.$$
Denote $F:=XZ$, then it follows from equality $ WX=\mathbb O_{p\times p}$  that $F=XYW.$
Hence
\be QFQ^\top=QXQ^\top QYQ^\top QWQ={\rm diag}(x_k,k=1,...,p)\bar Y {\rm diag}(w_k,k=1,...,p)\label{a-4}\ee
where $  \bar Y:=QYQ^\top$. On the other hand, equality (\ref{a-1}) implies that
\bea&\mathbb O_{p\times p}=Q(XY+YX)Q^\top=QXQ^\top QYQ^\top +QYQ^\top QXQ^\top\ \Longrightarrow\nonumber\\
& \mathbb O_{p\times p}={\rm diag}(x_k,k=1,...,p)\bar Y+\bar Y{\rm diag}(x_k,k=1,...,p)\ \Longrightarrow\label{x-1}\\
&{\rm diag}(x_k,k=1,...,p)\bar Y=-\bar Y{\rm diag}(x_k,k=1,...,p).\nonumber\eea
It follows from this equality, and equalities  $w_kx_k=0\ \forall \, k=1,...,p,$ and (\ref{a-4}) that
$$ QFQ^\top=-\bar Y{\rm diag}(x_k,k=1,...,p) {\rm diag}(w_k,k=1,...,p) =\mathbb O_{p\times p}.$$
Taking into account that ${\rm det}Q\not=0$ we conclude that $F=\mathbb O_{p\times p}.$

Now suppose that $(X+W)\in {\rm int}\,{\mathbb S}_+(p)$ and equality $Z=\mathbb O_{p\times p}$ holds true.
Then the set $P=\{1,...,p\} $ can be partition us follows
$$P_x:=\{k \in P:x_k>0\},\ P_y:=\{k \in P:y_k>0\}, \ P=P_x\cup P_y,\ P_x\cap P_y=\emptyset.$$
The equality $Z=\mathbb O_{p\times p}$ implies the equality $YW+WY=\mathbb O_{p\times p}$ wherefrom we obtain that
$$ \mathbb O_{p\times p}={\rm diag}(w_k,k=1,...,p)\bar Y+\bar Y{\rm diag}(w_k,k=1,...,p).$$
It follows from this equality and equality (\ref{x-1}) that
\bea& \bar Y_{kq}x_k+\bar Y_{kq}x_q=0 \ \forall  k \in P_x,\, q\in P_x;\ \bar Y_{kq}x_k=0\ \forall  k \in P_x,\, q\in P_y;\nonumber\\
&\bar Y_{kq}y_k+\bar Y_{kq}y_q=0 \ \forall  k \in P_y,\, q\in P_y;\ \bar Y_{kq}y_q=0\ \forall  k \in P_x,\, q\in P_y.\nonumber\eea
These equalities imply that $\bar Y_{kq}=0$ for all $k \in P$ and $q\in P$ and hence $Y=\mathbb O_{p\times p}$.
\end{proof}
\begin{proposition} \label{P01} Let  $X\in {\mathbb S}_+(p)$. Suppose that there exists $\bar t \in H(X):=\{t \in \mathbb R^p:Xt=0\}$
such that $\bar t>0.$ Then
$H(X)={\rm span} \{\tau(j), j \in J\},$
where $\tau(j), j \in J,$ are vertices of  the polyhedron $T_*(X):=H(X)\cap\{t \in \mathbb R^p:t\geq 0,\sum\limits_{k=1}^p t_k=1\}.$
\end{proposition}
\begin{proof} Note that by construction, $H(X)$ is a subspace, $T_*(X)$ is a convex bounded polyhedron, and  $\tau(j)\geq 0, j \in J.$

 Since $\tau(j) \in H(X),$ $j \in J,$ we have
${\rm span} \{\tau(j), j \in J\}\subset H(X).$
Let us show that $H(X)\subset {\rm span} \{\tau(j), j \in J\}.$

Let $t(i), i \in I_b,$ be a basis of the subspace $H(X)$. Then
$$\bar t=\sum\limits_{i\in I_b}\alpha_it(i), \; \sum\limits_{i\in I_b}|\alpha_i|>0.$$
It is evident that for sufficiently large $\theta >$ we have
$$\bar t(i):=t(i)+\theta \bar t >0,\ i \in I_b.$$
Consider an index $i_0\in I_b$ such that $\alpha_{i_0}\not=0$ and show that the vectors
$\bar t(i), \in I_b\setminus i_0,\; \bar t,$ are linearly independent.

Suppose the contrary:  there exist numbers $\beta_i,  i\in I_b\setminus i_0,$ $\beta_0$ such that
$\sum\limits_{i\in I_b\setminus i_0}\beta_i\bar t(i)+\beta_0\bar t=0,$ $ \sum\limits_{i\in I_b\setminus i_0}|\beta_i|+|\beta_0|>0.$ This implies that
 $$\sum\limits_{i\in I_b\setminus i_0}\beta_i t(i)+\bar \beta_0\bar t=0\mbox{ where } \bar \beta_0= \beta_0+\theta\sum\limits_{i\in I_b\setminus i_0}\beta_i.$$

 The following cases are possible:  A) $\bar \beta_0\not =0$,
\hspace{3mm}  B) $\bar \beta_0=0$, $ \sum\limits_{i\in I_b\setminus i_0}|\beta_i|\not =0.$

 Consider case A). Then
 $\bar t =\sum\limits_{i\in I_b\setminus i_0}\bar \beta_it(i)$ where $\bar \beta_i=-\beta_i/\bar \beta_0,\; i\in I_b\setminus i_0.$
 Hence
 $$\sum\limits_{i\in I_b}\alpha_it(i)=\sum\limits_{i\in I_b\setminus i_0}\bar \beta_it(i)\ \Longrightarrow\
 \alpha_{i_0}t(i_0)=\sum\limits_{i\in I_b\setminus i_0}(\bar \beta_i-\alpha_i)t(i)\ \mbox { where }\ \alpha_{i_0}\not=0.$$
 But this contradicts the assumption that the vectors
 $t(i), \ i \in I_b,$
 are linearly independent.

  In  case B), we have
$\sum\limits_{i\in I_b\setminus i_0} \beta_it(i)=0,\ \sum\limits_{i\in I_b\setminus i_0} |\beta_i|\not=0.$
But this contradicts  the assumption that the vectors  $t(i), \ i \in I_b,$
 are linearly independent.

Thus we have proved  that the vectors
$\bar t(i)\in H(X), i\in I_b\setminus i_0,\; \bar t $ are linearly independent.
 Let us set
 $$\hat t(i):=\bar t(i)/||\bar t(i)||_1, \; {i\in I_b\setminus i_0}; \ \hat t(i_0)=\bar t /||\bar t||_1.$$
By construction,   the  vectors $\hat t(i), i \in I_b,$ form a basis of the subspace $H(X)$ and satisfy the conditions:
 $\hat t(i)\in T_a(X),\; i \in I_b.$
 Notice that it follows from the latter inclusions that $\hat t(i)=\sum\limits_{j \in J}\alpha_{ij}\tau(j),$
 $\sum\limits_{j \in J}\alpha_{ij}=1,$ $\alpha_{ij}\geq 0,$ $j \in J.$
Then for  any $t \in H(X)$, we have $$t=\sum\limits_{i \in I_b}\beta_i\hat t(i)=
\sum\limits_{i \in I_b}\beta_i\sum\limits_{j \in J}\alpha_{ij}\tau(j)=\sum\limits_{j \in J}\bar \alpha_{j}\tau(j)
\mbox{ where } \bar \alpha_j=\sum\limits_{i \in I_b}\beta_i\alpha_{ij},\ j \in J.$$
  This implies that $t\! \in \!{\rm span}\{\tau(j), j \!\in\! J\}$ and consequently $ H(X)\!\subset {\rm span}\{\tau(j), j \!\in\! J\}$.
	\end{proof}

\begin{proposition}\label{PPA2} Let $X\in {\mathbb S}_+(p)$  and $Z\in {\mathbb S}(p)$ be such that
$ XZ=\mathbb O_{p\times p}$. Let $\tau(j), j \in J_b,$ be a basis of the subspace
$H(X):=\{t \in \mathbb R^m:Xt=0\}.$ Then the matrix $Z$ can be presented in the form
\be Z=\sum\limits_{(i,j)\in V(J_b)}\beta_{ij}(\tau(i)+\tau(j))(\tau(i)+\tau(j))^\top, \ \beta_{ij}\in \mathbb R \ \forall (i,j)\in V(J_b).\label{400}\ee
\end{proposition}
\begin{proof}If  $Z\in {\mathbb S}(p)$, then  there exists
 an orthogonal matrix $Q\in \mathbb R^{p\times p}$ with columns $q(k)$, $k=1,...,p,$ such that
\be Z=Q{\rm diag}(z_k,k=1,...,p)Q^\top=\sum\limits_{k=1}^pz_kq(k)q^\top(k)\label{401}\ee
 where $z_k,k=1,...,p,$ are the eigenvalues of $Z.$
Consequently
\bea&\!\!\mathbb O_{p\times p}=XZ=XQ{\rm diag}(z_k,k=1,...,p)Q^\top \, \Longrightarrow\, \mathbb O_{p\times p}=XQ{\rm diag}(z_k,k=1,...,p)\nonumber\\
&\Longrightarrow \ \ Xq(k)=0 \ \forall k\in I:=\{k\in\{1,...,p\}:z_k\not=0\}.\nonumber\eea
Thus, for $ k \in I,$ we obtain
$q(k)\in H(X) $ and consequently $ q(k)=\sum\limits_{j \in J_b}\alpha_j(k)\tau(j).$

This implies that, for $k \in I,$ we have
\bea &q(k)q^\top(k)=(\sum\limits_{j \in J_b}\alpha_j(k)\tau(j))(\sum\limits_{i \in J_b}\alpha_i(k)\tau(j))^\top\nonumber\\
&=\sum\limits_{j \in J_b}\alpha^2_j(k)\tau(j)(\tau(j))^\top +\sum\limits_{(i,j)\in V_*(J_b)}\alpha_i(k)\alpha_j(k)
(\tau(j)(\tau(i))^\top+\tau(i)(\tau(j))^\top)\nonumber\eea
where $V_*(J_b)=V(J_b)\setminus \{(j,j), j \in J_b\}.$
Taking into account that
$$\tau(j)(\tau(i))^\top+\tau(i)(\tau(j))^\top=(\tau(i)+\tau(j))(\tau(i)+\tau(j))^\top-\tau(j)(\tau(j))^\top-\tau(i)(\tau(i))^\top$$
we obtain that
$q(k)q^\top(k)=\sum\limits_{(i,j) \in V(J_b)}\bar \beta_{ij}(k) (\tau(i)+\tau(j))(\tau(i)+\tau(j))^\top $  with some
$\bar \beta_{ij}(k)\in \mathbb R, \ (i,j) \in V(J_b).$
It follows from the latter equality and (\ref{401}) that presentation (\ref{400}) holds true %Consequently
 with $\beta_{ij}=\sum\limits_{k\in I}z_k \bar \beta_{ij}(k), \ (i,j) \in V(J_b).$
\end{proof}

\end{document}